\documentclass[final,3p]{elsarticle}
\usepackage{color}
\usepackage{amssymb}
\usepackage{amsmath}
\usepackage{booktabs}
\usepackage{multirow}
\usepackage{tabularx}
\usepackage{tabu}
\usepackage{stmaryrd}
\usepackage{epstopdf}
\usepackage{subfigure} 
\usepackage{graphicx}
\usepackage{caption}
\usepackage{amsthm} 
\usepackage{amssymb} 
\usepackage{mathrsfs}

\newcommand{\mbf}{\mathbf}

\newcommand{\diag}{\text{diag}}

\numberwithin{equation}{section}
\newtheorem{thm}{Theorem}[section]

\newtheorem{rem}[thm]{Remark}
\newtheorem{ex}[thm]{Example}
\allowdisplaybreaks[4]

\biboptions{}
 \journal{*}
\begin{document}
	
\begin{frontmatter}
\title{Novel semi-explicit symplectic schemes for nonseparable stochastic Hamiltonian systems}

\author[ad1,ad2]{Jialin Hong}
\ead{hjl@lsec.cc.ac.cn} 

\author[ad3]{Baohui Hou\corref{cof1}}
\ead{houbaohui@shu.edu.cn} 

\author[ad4]{Liying Sun}
\ead{liyingsun@lsec.cc.ac.cn} 

\address[ad1]{Institute of Computational Mathematics and Scientific/Engineering Computing, Academy of Mathematics and Systems Science, Chinese Academy of Sciences, Beijing 100190, China}
\address[ad2]{School of Mathematical Sciences, University of Chinese Academy of Sciences, Beijing 100049, China}
\address[ad3]{Department of Mathematics, Shanghai University, Shanghai 200444, China}
\cortext[cof1]{Corresponding author}
\address[ad4]{Academy for Multidisciplinary Studies, Capital Normal University, Beijing 100048, China}

\begin{abstract}
In this manuscript, we propose efficient stochastic semi-explicit symplectic schemes tailored for nonseparable stochastic Hamiltonian systems (SHSs). These semi-explicit symplectic schemes are constructed by introducing  augmented Hamiltonians and using symmetric projection. In the case of the artificial restraint in augmented Hamiltonians being zero, the proposed schemes also preserve quadratic invariants, making them suitable for developing semi-explicit charge-preserved multi-symplectic
schemes for stochastic cubic Schr\"odinger equations with multiplicative noise. Through numerical experiments that validate theoretical results, we demonstrate that the proposed stochastic semi-explicit symplectic scheme, which features a straightforward Newton iteration solver, outperforms the traditional stochastic midpoint scheme in terms of effectiveness and accuracy.
\end{abstract}	


\end{frontmatter}

\section{Introduction} 
The SHS is a fundamental mathematical model used to describe the dynamics of physical systems subject to both deterministic and random influences. It extends the deterministic  Hamiltonian formalism by incorporating stochastic terms that account for random fluctuations or noise in the system. SHSs are commonly used in physics, chemistry, and other fields to model complex systems where both deterministic and stochastic forces play a role in the system’s behavior. By considering the interplay between deterministic dynamics and random fluctuations, SHSs provide a powerful tool for understanding and analyzing the behavior of a wide range of systems in the presence of uncertainty  (see \cite{CHJS2021,DADG2023,MR1201269,GT1987} and references therein).  
In this manuscript, we consider the SHS  
\begin{equation}
\label{eq;SHS}
\left\{\begin{aligned}
&d X=  \frac{\partial H_0(X, Y)}{\partial Y} d t + \sum\limits_{r=1}^m\frac{\partial H_r(X, Y)}{\partial Y} \circ d W_r,\quad X(0)=x^0, \\
&d Y= -\frac{\partial H_0(X, Y)}{\partial X} d t - \sum\limits_{r=1}^m\frac{\partial H_r(X, Y)}{\partial X} \circ d W_r,\quad Y(0)=y^0,
\end{aligned}\right.
\end{equation}
where $t \in(0, T], x^0, y^0$ are $d$-dimensional column vectors, $H_{r} \in \mathbf{C}^{\infty}\left(\mathbb{R}^{2 d}, \mathbb{R}\right)$ are nonseparable Hamiltonians, and $W_r(\cdot)$ are independent standard Wiener processes for $r=0,1,\ldots,m.$ 
Under appropriate conditions, the SHS \eqref{eq;SHS} admits a unique exact solution. 
The Hamiltonian $H_r,$ $r\in\{0,1,\ldots,m\}$ and \eqref{eq;SHS}  are said to be separable if $H_r$ can be written as $H_r(q, p)=\mathcal{K}_r(p)+\mathcal{L}_r(q)$ with some functions $\mathcal{K}_r$ and $\mathcal{L}_r$, and nonseparable otherwise. 
The stochastic symplectic structure of the stochastic flow $\phi_{t}:(x^0, y^0) \mapsto(X(t), Y(t))$ of \eqref{eq;SHS} can be characterized as
\begin{align*}
\mathrm{d} X(t) \wedge \mathrm{d} Y(t)=\mathrm{d} x^0 \wedge \mathrm{d} y^0, \quad a. s.
\end{align*}
for any $t \in[0, T]$, where $`{\rm d}'$ denotes the differential with respect to the initial value. 
This geometric feature ensures that phase space volume is almost surely conserved along trajectories. 

When numerically simulating a SHS, it is essential to ensure that the numerical scheme preserves the stochastic symplectic geometric structure. 
Stochastic symplectic schemes maintain the accuracy and 
demonstrate good energy behavior in long-time simulations, compared with nonsymplectic ones.  Recently, much attention has been paid to both the construction and analysis of stochastic symplectic schemes for SHSs (see \cite{Anton1,Anton2,CHJ2023LNM,HT2016p1,HS2002book,HSW2017,HW2019LNM, MRT2002p1,MRT2002p2,WWC2022} and references therein. 
For instance, stochastic symplectic schemes were initially introduced in \cite{MRT2002p1,MRT2002p2}. A methodology for constructing stochastic symplectic schemes using the stochastic generating function theory, which involves approximately solving a corresponding stochastic Hamilton--Jacobi partial differential equation, was suggested in \cite{Anton1,Anton2,HS2002book,HSW2017} and references therein. 
The authors in \cite{HT2018} present a general framework for constructing stochastic Galerkin variational integrators, and prove the symplecticity of such integrators for stochastic Hamiltonian systems.
Stochastic symplectic schemes created through composition methods were proposed and studied in \cite{Misawa2010}. 
The separability of  Hamiltonians plays a significant role in the development of symplectic schemes for the SHS.  
In some cases, the separability of  Hamiltonians can transform implicit stochastic symplectic schemes into explicit ones, and it can also simplify the application of the splitting strategy. 
The reason is that the stochastic flows associated with separable Hamiltonians are exactly solvable, making the integration process more manageable. 
However, constructing efficient stochastic symplectic schemes for nonseparable SHSs presents a notable challenge, and would be very useful in numerical simulations. 
Although explicit stochastic symplectic schemes can be derived for certain categories of Hamiltonian systems, these schemes are tailored to specific forms of Hamiltonians and can not be generalized generally to arbitrary nonseparable cases in a simple manner.  

To address this issue, we introduce augmented separable Hamiltonians and establish a 4$d$-dimensional SHS that provides two copies of the original system. 
This SHS can be decomposed into three subsystems that are solvable explicitly while retaining stochastic symplecticity. 
Based on the splitting strategy and the exact solutions of subsystems, we present stochastic symplectic schemes conserving symplectic structure in the extended phase space, rather than the original phase space. 
Despite the explicit nature of these stochastic symplectic schemes, the numerical solution does not remain within the original phase space copies. 
As a consequence, we combine with the symmetric projection and then propose stochastic semi-explicit symplectic schemes for a broader range of nonseparable SHSs. 
In particular, the proposed stochastic symplectic schemes with specific extended phase space symplectic integrators are able to preserve quadratic invariants of \eqref{eq;SHS}. 
Numerical results reveal that the proposed stochastic semi-explicit symplectic scheme is more efficient than the traditional stochastic midpoint scheme. 
This enhanced efficiency is attributed to the simplified nature of the Newton iteration employed in solving the stochastic semi-explicit symplectic scheme. 
Furthermore, by applying a stochastic semi-explicit symplectic scheme in temporal direction and finite difference in spatial direction to discretizing stochastic cubic Schr\"odinger equations with multiplicative noise, we obtain a novel stochastic multi-symplectic scheme, which also preserves the charge conservation law.

The manuscript is organized as follows. In Section 2, we introduce stochastic semi-explicit symplectic schemes for nonseparable SHSs via the extended phase space idea and symmetric projection.
In Section 3, we present semi-explicit fully-discrete schemes preserving both the stochastic multi-symplectic structure and charge conservation law for stochastic cubic Schr\"odinger equations with multiplicative noise.

\section{Stochastic semi-explicit symplectic schemes for SHS}
Motivated by the extended phase space idea from \cite{jayawardana, TaoPRE2016},  we attempt to achieve explicit symplectic schemes inheriting the stochastic symplecticity of the stochastic flow for SHS \eqref{eq;SHS}. 
To this end, we consider augmented Hamiltonians
\begin{align*}
\tilde{H}_i(X,U,Y,V):=H_{i}(X,V)+H_{i}(U,Y)+\gamma_i  \|X-U\|^2+\gamma_i\|Y-V\|^2
\end{align*}
where $i=0,1,\ldots,m,$ $\gamma_i  \|X-U\|^2+\gamma_i\|Y-V\|^2$ is an artificial restraint with $\gamma_i$ being a constant that controls the binding of the two copies. 
By exploiting augmented Hamiltonians and $dW_0(t):=dt,$ we obtain a novel extended SHS
\begin{equation}
\label{eq;ExtendSHS}
\left\{\begin{aligned}
&d X=  \sum\limits_{r=0}^m\bigg (\frac{\partial H_r(U, Y)}{\partial Y} +2\gamma_r(Y-V)\bigg)\circ d W_r,\\
&d U= \sum\limits_{r=0}^m\bigg (\frac{\partial H_r(X, V)}{\partial V}-2\gamma_r(Y-V) \bigg)\circ d W_r, \\
&d Y= -\sum\limits_{r=0}^m\bigg (  \frac{\partial H_r(X, V)}{\partial X}+2\gamma_r(X-U)\bigg)\circ d W_r,\\
&d V= -\sum\limits_{r=0}^m\bigg ( \frac{\partial H_r(U, Y)}{\partial U} -2\gamma_r(X-U)\bigg)\circ d W_r.
\end{aligned}\right.
\end{equation}
with the symplectic form $\widetilde{\omega} = \mathrm{d} X(t) \wedge \mathrm{d} Y(t) + \mathrm{d} U(t) \wedge \mathrm{d} V(t).$    
If we suppose the initial condition $(X(0), U(0),Y(0),V(0))=(x^0,x^0,y^0,y^0)$, 
then the solution of \eqref{eq;ExtendSHS} satisfies $(X(t), Y(t)) = (U(t), V(t))$ for any $t\in [0,T]$ and $(X(t),Y(t))$ coincides with the solution to \eqref{eq;SHS}. 
In other words, the system \eqref{eq;ExtendSHS} gives two copies of the original system \eqref{eq;SHS}. 
Moreover, given a linear invariant of SHS \eqref{eq;SHS} as follows
$$\mathbb I_1(Z):=a^\top Z,\quad a=\left(a_x^\top, a_y^\top\right)^\top \in \mathbb{R}^{2d},\quad a_x, a_y\in \mathbb{R}^d$$ with $Z=(X^\top,Y^\top)^\top,$ 
it can be verified that 
$$\hat{L}_{1}(\zeta):=\hat a^\top\zeta,\quad \zeta= (X^\top,U^\top,Y^\top,V^\top)^\top,\quad \hat{a}=\frac{1}{2}\left(a_x, a_x, a_y, a_y\right) \in \mathbb{R}^{4 d}$$ is a linear invariant of the extended Hamiltonian system \eqref{eq;ExtendSHS} since
\begin{align}
\label{conditon1}
a^\top J_{2d}\nabla H_r(X,V)+a^\top J_{2d}\nabla H_r(U,Y)=0,\quad r=0,1,\ldots,m
\end{align}
with $J_{2d}$ being  standard symplectic matrix. 

The salient feature of \eqref{eq;ExtendSHS} is that $X$ and $Y$ of the Hamiltonian of \eqref{eq;ExtendSHS} is  separable.  
Respectively, denote
by $\mathcal{F}_{\Delta t}^1$, $\mathcal{F}_{\Delta t}^2$ ,$\mathcal{F}_{\Delta t}^3$ the stochastic flow of $H_r(X,V), H_{r}(U,Y),$ $\gamma_r \|X-U\|^2+\gamma_r\|Y-V\|^2$ with $r=0,1,\ldots, m.$  
Exact expressions of these stochastic flows can be explicitly obtained as
\begin{align}
&\mathcal{F}_{\Delta t}^1:
\begin{pmatrix}
   X^n\\
   U^n\\
   Y^n\\
   V^n 
\end{pmatrix}
\rightarrow
\begin{pmatrix}
    X^n\\
    U^n + \sum\limits_{r=0}^m\frac{\partial H_r}{\partial V}(X^n, V^n)\Delta W_r^n\\
    Y^n -\sum\limits_{r=0}^m  \frac{\partial H_r}{\partial X}(X^n, V^n)\Delta W_r^n\\
    V^n
\end{pmatrix},\\
&\mathcal{F}_{\Delta t}^2:
\begin{pmatrix}
   X^n\\
   U^n\\
   Y^n\\
   V^n 
\end{pmatrix}
\rightarrow
\begin{pmatrix}
X^n + \sum\limits_{r=0}^m\frac{\partial H_r}{\partial Y}(U^n, Y^n)\Delta W_r^n\\
U^n\\
Y^n\\
V^n - \sum\limits_{r=0}^m\frac{\partial H_r}{\partial U}(U^n, Y^n)\Delta W_r^n
\end{pmatrix},\\
&\mathcal{F}_{\Delta t}^3:
\begin{pmatrix}
   X^n\\
   U^n\\
   Y^n\\
   V^n 
\end{pmatrix}
\rightarrow
\frac 12\begin{pmatrix}
X^n+U^n+\cos(\Theta)(X^n-U^n)+\sin(\Theta)(Y^n-V^n)\\
X^n+U^n-\cos(\Theta)(X^n-U^n)-\sin(\Theta)(Y^n-V^n)\\
Y^n+V^n-\sin(\Theta)(X^n-U^n)+\cos(\Theta)(Y^n-V^n)\\
Y^n+V^n+\sin(\Theta)(X^n-U^n)-\cos(\Theta)(Y^n-V^n)
\end{pmatrix},
\end{align}
where $\Delta t$ is the temporal step size and $\Theta=4\sum\limits_{r=0}^m\gamma_r\Delta W_r^n$ with $\Delta W_r^n=W_r(t_n+\Delta t)-W_r(t_n),$ $r=0,1,\ldots,m,$ and $t_n=n\Delta t$ for $n\in\{0,1,\ldots,N\}$ and $n=\frac T{\Delta t}.$  
As a result, explicit phase space integrators  for \eqref{eq;ExtendSHS} can be constructed based on the splitting strategy,  such as
\begin{equation}
\label{split1}
\mathcal{F}_{\Delta t}: =\mathcal{F}_{\Delta t }^1 \star \mathcal{F}_{\Delta t }^2\star \mathcal{F}_{\Delta t }^3,
\end{equation}
and the Strang splitting
\begin{equation}
\label{strang}
\mathcal{F}_{\Delta t}: =\mathcal{F}_{\Delta t / 2}^1 \star \mathcal{F}_{\Delta t / 2}^2\star
\mathcal{F}_{\Delta t }^3
\star\mathcal{F}_{\Delta t / 2}^2\star
\mathcal{F}_{\Delta t / 2}^1.
\end{equation}
Since the composition of symplectic transformations is also symplectic, 
the explicit phase space integrator $\widetilde \xi^{n+1}=\mathcal{F}_{\Delta t}(\widetilde{\xi}^{n})$ with $\widetilde{\xi}^{n}= ((\widetilde{X}^{n})^\top, (\widetilde{U}^{n})^\top, (\widetilde{Y}^{n})^\top, (\widetilde{V}^{n})^\top)^\top$ 
preserves the symplectic structure almost surely, i.e.,
$$\mathrm{d}\widetilde{X}^{n+1} \wedge \mathrm{d} \widetilde{Y}^{n+1}  + \mathrm{d}\widetilde{U}^{n+1} \wedge \mathrm{d} \widetilde{V}^{n+1}
= \mathrm{d}\widetilde{X}^{n} \wedge \mathrm{d} \widetilde{Y}^{n}  + \mathrm{d}\widetilde{U}^{n} \wedge \mathrm{d} \widetilde{V}^{n},\quad a.s.$$  
where $n\in\{0,1,\ldots,N-1\}.$ 
However, the above scheme is symplectic in the extended phase space but not in the original phase space. 
To this issue, we turn to introduce the symmetric projection (\cite{HLW2006}), which ensures $(\widetilde U^n, \widetilde V^n) = (\widetilde X^n, \widetilde Y^n)$ for any $n\geq 0$. 
That is to say, it eliminates the problematic defect $(\widetilde X^n-\widetilde U^n, \widetilde Y^n-\widetilde V^n)$ and defines a discrete flow $\widetilde{\mathcal{F}}_{\Delta t}:  (X^n, Y^n)\rightarrow  (X^{n+1}, Y^{n+1})$ in the original phase space for $n\in\{0,1,\ldots,N-1\}$.

\noindent$\bf\hrulefill$

\noindent {\textbf {Scheme 2.1}} Semi-explicit symplectic scheme based on symmetric projection 

\vspace{0.5em}
\noindent Given an extended phase space scheme $\mathcal{F}_{\Delta t}$ and $Z^n =\left(X^n,Y^n\right),$ $n\in\{0,1,\ldots,N-1\},$ 
\begin{itemize}
\item[1.] Set $\xi^n:=\left((X^n)^\top, (X^n)^\top, (Y^n)^\top, (Y^n)^\top\right)\in \mathcal N:=\ker(A)$ and $\widetilde{\xi}^n:=\xi^n+A^\top \lambda,$  where \begin{equation*}
A=\left[\begin{array}{cccc}
I_d & -I_d & 0 &0\\
0 & 0 & I_d & -I_d 
\end{array}\right].
\end{equation*}
\item[2.]  Find $\lambda \in \mathbb{R}^{2d}$ such that ${\xi}^{n+1}:= \mathcal{F}_{\Delta t}(\widetilde{\xi}^{n})+A^\top \lambda \in \mathcal{N}.$
\item [3.]Let $\widetilde{\xi}^{n+1}=\mathcal F_{\Delta t}(\widetilde{\xi}^{n}).$ 
\item [4.] Let $\xi^{n+1}=\left((X^{n+1})^\top, (X^{n+1})^\top, (Y^{n+1})^\top, (Y^{n+1})^\top\right)^\top:=\widetilde{\xi}^{n+1}+A^\top \lambda.$ 
\item [5.] $Z^{n+1}:=\left((X^{n+1})^\top, (Y^{n+1})^\top\right)^\top.$
\end{itemize}
\noindent$\bf\hrulefill$

To make sure that $\widetilde{\mathcal{F}}_{\Delta t}:  (X^n, Y^n)\rightarrow  (X^{n+1}, Y^{n+1}),$ $n\in\{0,1,\ldots,N-1\},$  exists almost surely, it suffices to prove that for a given $\xi^n \in \mathcal{N}$, there exist $\xi^{n+1} \in \mathcal{N}$ and $\lambda \in \mathbb{R}^{2 d}$ satisfying 
$$
\xi^{n+1}
=
\mathcal{F}_{\Delta t}\left(\xi^n+A^\top \lambda\right)+A^\top \lambda,\quad A\xi^{n+1}=0.
$$
Now we define $\varphi :\left(\mathbb{R} \times \mathbb{R}^{4d}\right) \times\left(\mathbb{R}^{4d}\times \mathbb{R}^{2 d}\right) \rightarrow \mathbb{R}^{4d} \times \mathbb{R}^{2 d}$ by
\begin{align*}
\varphi \left(\left(\Delta t, \xi^n\right),(\xi^{n+1}, \lambda)\right):=\left[\begin{array}{c}
\xi^{n+1}-\mathcal{F}_{\Delta t}\left(\xi^n+A^\top \lambda\right)-A^\top \lambda\\
A \xi^{n+1}
\end{array}\right].
\end{align*}
It can be verified that $AA^\top=2I_{2d}$ and the Jacobian matrix
\begin{equation*}
\frac{\partial \varphi}{\partial(\xi^{n+1}, \lambda)}\left(\left(0, \xi^n\right),\left(\xi^n, 0\right)\right)=\left[\begin{array}{cc}
I_{4 d} & -2 A^\top \\
A & 0
\end{array}\right]
\end{equation*}
is invertible.  
Based on the implicit function theorem, we derive that there exist a neighborhood $\mathcal{B} \subset \mathbb{R}^{4d}$ of $\xi^n,$  $\varepsilon>0$, and mappings
$\Psi_{(\cdot)}:(-\varepsilon, \varepsilon) \times \mathcal{B} \rightarrow \mathbb{R}^{4 d}$ and $\rho_{(\cdot)}:(-\varepsilon, \varepsilon) \times B \rightarrow \mathbb{R}^{2 d}$ 
such that for any $\left(\Delta t, \xi^n\right) \in(-\varepsilon, \varepsilon) \times \mathcal{B}$,
\begin{equation*}
\varphi \left(\left(\Delta t, \xi^n\right),\left(\Psi_{\Delta t}\left(\xi^n\right), \rho_{\Delta t}\left(\xi^n\right)\right)\right)=0. 
\end{equation*}
Thus, there exists
$\widetilde{\mathcal{F}}_{\Delta t}: \mathcal B \rightarrow  \mathbb{R}^{2d}$ such that 
$Z^{n+1}=\widetilde{\mathcal{F}}_{\Delta t}(Z^{n}),$ $n\in\{0,1,\ldots,N-1\}.$ 
In the following, we show that $\widetilde{\mathcal{F}}_{\Delta t}$ is symplectic almost surely. 

\begin{thm}
\label{tm:finite symplectic}
If the extended phase space scheme based on $\mathcal{F}_{\Delta t}$ is symplectic, then the stochastic semi-explicit  scheme based on $\widetilde{\mathcal{F}}_{\Delta t}$ preserves symplectic structure almost surely, i.e., 
\begin{align*}
\mathrm{d} X^{n+1} \wedge \mathrm{d} Y^{n+1} =\mathrm{d} X^n \wedge \mathrm{d}Y^n, \quad a. s.,
\end{align*}
where $n\in\{0,1,\ldots,N-1\}$.
\end{thm}
\begin{proof}
For any $n\in\{0,1,\ldots,N-1\}.$ 
Assume that $\widetilde{\xi}^{n}= ((\widetilde{X}^{n})^\top, (\widetilde{U}^{n})^\top, (\widetilde{Y}^{n})^\top,\\ (\widetilde{V}^{n})^\top)^\top$, $\lambda=\left( \lambda_1^\top,  \lambda_2^\top\right)^\top$, where $ \lambda_1, \lambda_2 :\mathbb R^{4d} \rightarrow \mathbb{R}^d$,
then  
\begin{align*}
\widetilde{\xi}^{n+1}&= \left((\widetilde{X}^{n+1})^\top, (\widetilde{U}^{n+1})^\top, (\widetilde{Y}^{n+1})^\top, (\widetilde{V}^{n+1})^\top\right)^\top\\
&
=\left((X^{n+1}-\lambda_1)^\top, (X^{n+1}+\lambda_1)^\top, (Y^{n+1}-\lambda_2)^\top, (Y^{n+1}+\lambda_2)^\top\right)^\top,\\
\widetilde{X}^{n} - &\widetilde{U}^{n} = 2\lambda_1= \widetilde{U}^{n+1} - \widetilde{X}^{n+1},
~~\widetilde{Y}^{n} - \widetilde{V}^{n} = 2\lambda_2= \widetilde{V}^{n+1} - \widetilde{Y}^{n+1}.
\end{align*}
Then $X^{n+1}= \frac{1}{2}(\widetilde{X}^{n+1}+\widetilde{U}^{n+1} ), Y^{n+1}= \frac{1}{2}(\widetilde{Y}^{n+1}+\widetilde{V}^{n+1} )$ and 
\begin{align*}
\nonumber
&4\mathrm{d} X^{n+1} \wedge \mathrm{d} Y^{n+1} \\
= & \mathrm{d}\widetilde{X}^{n+1} \wedge \mathrm{d} \widetilde{Y}^{n+1}  + \mathrm{d}\widetilde{X}^{n+1} \wedge \mathrm{d} \widetilde{V}^{n+1}+ \mathrm{d}\widetilde{U}^{n+1} \wedge \mathrm{d} \widetilde{Y}^{n+1}+ \mathrm{d}\widetilde{U}^{n+1} \wedge \mathrm{d} \widetilde{V}^{n+1}.
\end{align*}
According to the symplecticity of $\mathcal F_{\Delta t },$ we have  
\begin{align*}
\nonumber
&\mathrm{d}\widetilde{X}^{n+1} \wedge \mathrm{d} \widetilde{Y}^{n+1}  + \mathrm{d}\widetilde{U}^{n+1} \wedge \mathrm{d} \widetilde{V}^{n+1}
= \mathrm{d}\widetilde{X}^{n} \wedge \mathrm{d} \widetilde{Y}^{n}  + \mathrm{d}\widetilde{U}^{n} \wedge \mathrm{d} \widetilde{V}^{n}\\
\nonumber
=&\mathrm{d}(X^n+\lambda_1)\wedge\mathrm{d}(Y^n+\lambda_2)+\mathrm{d}(X^n-\lambda_1)\wedge\mathrm{d}(Y^n-\lambda_2)\\
=&2\mathrm{d}X^n\wedge\mathrm{d}Y^n+2\mathrm{d}\lambda_1\wedge\mathrm{d}\lambda_2.
\end{align*} 
Moreover, since 
\begin{equation}\label{eq;3-3}
 \mathrm{d}(\widetilde{U}^{n+1} - \widetilde{X}^{n+1}) \wedge \mathrm{d}(\widetilde{V}^{n+1} - \widetilde{Y}^{n+1})
 = \mathrm{d}(\widetilde{X}^{n} - \widetilde{U}^{n} ) \wedge \mathrm{d}(\widetilde{Y}^{n} - \widetilde{V}^{n}),
\end{equation}
we obtain 
\begin{align*}
&\mathrm{d}\widetilde{X}^{n+1} \wedge \mathrm{d} \widetilde{Y}^{n+1}  + \mathrm{d}\widetilde{U}^{n+1} \wedge \mathrm{d} \widetilde{V}^{n+1}-
\mathrm{d}\widetilde{U}^{n+1} \wedge \mathrm{d}\widetilde{Y}^{n+1}-  \mathrm{d}\widetilde{X}^{n+1} \wedge \mathrm{d}\widetilde{V}^{n+1}\\
=&2\mathrm{d}X^n\wedge\mathrm{d}Y^n+2\mathrm{d}\lambda_1\wedge\mathrm{d}\lambda_2-\mathrm{d}(X^n-\lambda_1)\wedge\mathrm{d}(Y^n+\lambda_2)-\mathrm{d}(X^n+\lambda_1)\wedge\mathrm{d}(Y^n-\lambda_2)\\
=&2\mathrm{d}X^n\wedge\mathrm{d}Y^n+2\mathrm{d}\lambda_1\wedge\mathrm{d}\lambda_2-2\mathrm{d}X^n\wedge\mathrm{d}Y^n+2\mathrm{d}\lambda_1\wedge\mathrm{d}\lambda_2=4\mathrm{d}\lambda_1\wedge\mathrm{d}\lambda_2.
\end{align*}
As a result, 
\begin{align*}
4\mathrm{d} X^{n+1} \wedge \mathrm{d} Y^{n+1} 
= & 2\mathrm{d}\widetilde{X}^{n+1} \wedge \mathrm{d} \widetilde{Y}^{n+1} +2\mathrm{d}\widetilde{U}^{n+1} \wedge \mathrm{d} \widetilde{V}^{n+1}-4\mathrm{d}\lambda_1\wedge\mathrm{d}\lambda_2\\
=&4\mathrm{d}X^n\wedge\mathrm{d}Y^n+4\mathrm{d}\lambda_1\wedge\mathrm{d}\lambda_2-4\mathrm{d}\lambda_1\wedge\mathrm{d}\lambda_2,
\end{align*}
which yields the result.
\end{proof}

The above theorem illustrates the geometric property of the proposed scheme $\widetilde{\mathcal F}_{\Delta t}$, now we turn to the preservation of quadratic invariants for the SHS.

\begin{thm}
For any $n\in\{0,1,\ldots,N-1\}$, the proposed stochastic semi-explicit symplectic scheme $(X^{n+1}, Y^{n+1})=\widetilde{\mathcal{F}}_{\Delta t}(X^n, Y^n)$ with $\gamma_r=0,$ $r=0,1,\ldots,m,$ 
preserves the quadratic invariant of SHS  \eqref{eq;SHS} as follows
$$
Q_\kappa(Z):=\frac{1}{2} Z^\top \kappa Z=\frac{1}{2} X^\top \kappa_{11} X+X^\top \kappa_{12} Y+\frac{1}{2} Y^\top \kappa_{22} Y, 
$$
where $\kappa=\begin{bmatrix}
    \kappa_{11} &\kappa_{12}\\
    \kappa_{12}^\top & \kappa_{22}
\end{bmatrix}$ is a $2d$-dimensional symmetric and positive definite matrix. 
 
\end{thm}
\begin{proof} 
Denoting 
\begin{align*}
\kappa_1=\begin{bmatrix}
\kappa_{12} & -\kappa_{11}\\
\kappa_{22} &-\kappa_{12}^\top
\end{bmatrix},\quad 
\hat \kappa=
\begin{bmatrix}
0 & \kappa_{11} &\kappa_{12} & 0\\
\kappa_{11} & 0 & 0 &\kappa_{12}\\
\kappa_{12}^\top & 0 & 0 &\kappa_{22}\\
0 & \kappa_{12}^\top & \kappa_{22} & 0
\end{bmatrix}, \quad
\widetilde \kappa=
\begin{bmatrix}
\kappa_{11} & 0 & 0 &\kappa_{12}\\
0 & \kappa_{11} &\kappa_{12} & 0\\
0 & \kappa_{12}^\top & \kappa_{22} & 0\\
\kappa_{12}^\top & 0 & 0 &\kappa_{22}
\end{bmatrix}, 
\end{align*}
we have
$
\frac 12(X^\top,Y^\top)^\top\kappa_1
\begin{bmatrix}
\frac{\partial H_r(X,Y)}{\partial X}\\
\frac{\partial H_r(X,Y)}{\partial Y}
\end{bmatrix}=0,$ $r=0,1,\ldots,m,
$
which means 
\begin{align*}
(X^\top,V^\top)^\top\kappa_1
\begin{bmatrix}
\frac{\partial H_r(X,V)}{\partial X}\\
\frac{\partial H_r(X,V)}{\partial V}
\end{bmatrix}+
(U^\top,Y^\top)^\top\kappa_1
\begin{bmatrix}
\frac{\partial H_r(U,Y)}{\partial U}\\
\frac{\partial H_r(U,Y)}{\partial Y}
\end{bmatrix}
=0,
\end{align*}
equivalently, 
\begin{align*}
\zeta
\hat\kappa
J_{2d}
\bigg[\frac{\partial H_r(X,V)}{\partial X}^\top,\frac{\partial H_r(U,Y)}{\partial U} ^\top,\frac{\partial H_r(U,Y)}{\partial Y}^\top,\frac{\partial H_r(X,V)}{\partial V}^\top\bigg]^\top
=0,
\end{align*}
where $\zeta= (X^\top,U^\top,Y^\top,V^\top)^\top,$ $r=0,1,\ldots,m.$
As a result, $Q_\kappa(Z)$ is a quadratic invariant of SHS \eqref{eq;SHS} if and only if
$$
\hat{Q}_\kappa(\zeta):=\frac{1}{2}\left(X^\top\kappa_{11} U+X^\top \kappa_{12} Y+U^\top \kappa_{12} V+Y^\top \kappa_{22} V\right)=\frac 14\zeta^\top\hat\kappa\zeta
$$
is a quadratic invariant of the extended SHS \eqref{eq;ExtendSHS}. 
Since $\xi^{n+1}=\widetilde{\xi}^{n+1}+A^\top \lambda$ and $\xi^{n}=\widetilde{\xi}^{n}-A^\top \lambda$ with $\lambda=-\frac 12 A\widetilde{\xi}^{n+1}=\frac 12 A\widetilde{\xi}^{n}$,  we obtain
\begin{align*}
&Q_\kappa\left(Z^{n+1}\right)-Q_\kappa\left(Z^{n}\right)=\frac 14(\xi^{n+1})^\top\widetilde \kappa\xi^{n+1}-\frac 14(\xi^{n})^\top\widetilde \kappa\xi^{n}\\
=&\frac 14 (\widetilde \xi^{n+1}+A^\top\lambda)^\top \widetilde\kappa(\widetilde \xi^{n+1}+A^\top\lambda) -\frac 14 (\widetilde \xi^{n}-A^\top\lambda)^\top \widetilde\kappa(\widetilde \xi^{n}-A^\top\lambda)\\
=&\frac 14(\widetilde{\xi}^{n+1})^\top\widetilde\kappa(I_{4d}-A^\top A)\widetilde{\xi}^{n+1}-\frac 14 (\widetilde{\xi}^{n})^\top\widetilde\kappa(I_{4d}-A^\top A)\widetilde{\xi}^{n}\\
=&\hat{Q}_\kappa(\widetilde{\xi}^{n+1})-\hat{Q}_\kappa(\widetilde{\xi}^{n}),
\end{align*} 
where $\widetilde{\xi}^{n+1}=\mathcal F_{\Delta t}(\widetilde{\xi}^{n}).$ 
\begin{itemize}
    \item If $\mathcal F_{\Delta t}=\mathcal F^1_{\Delta t} \star \mathcal F^2_{\Delta t}$, 
\begin{align*}
\mathcal F_{\Delta t}: \widetilde\xi^n=
\begin{pmatrix}
    \widetilde X^n\\
    \widetilde U^n\\
    \widetilde Y^n\\
    \widetilde V^n
\end{pmatrix}
\rightarrow
\begin{pmatrix}
\widetilde X^n+\sum\limits_{r=0}^m
\frac{\partial H_r(\hat  U^n,\hat Y^n)}{\partial Y}\Delta W_r^n\\
\widetilde U^n+\sum\limits_{r=0}^m
\frac{\partial H_r(\widetilde X^n,\widetilde V^n)}{\partial V}\Delta W_r^n\\
\widetilde Y^n-\sum\limits_{r=0}^m\frac{\partial H_r(\widetilde X^n,\widetilde V^n)}{\partial X}\Delta W_r^n\\
\widetilde V^n-\sum\limits_{r=0}^m
\frac{\partial H_r(\hat  U^n,\hat Y^n)}{\partial U}\Delta W_r^n
\end{pmatrix}
=\widetilde\xi^{n+1},
\end{align*}
where $\hat U^n=\widetilde U^n+\sum\limits_{r=0}^m
\frac{\partial H_r(\widetilde X^n,\widetilde V^n)}{\partial V}\Delta W_r^n,$ $\hat Y^n=\widetilde Y^n+\sum\limits_{r=0}^m
\frac{\partial H_r(\widetilde X^n,\widetilde V^n)}{\partial X}\Delta W_r^n,$  $n\in\{0,1,\ldots,N-1\}.$ 
\item For the Strang splitting $\mathcal{F}_{\Delta t}: =\mathcal{F}_{\Delta t / 2}^1 \star \mathcal{F}_{\Delta t}^2\star \mathcal{F}_{\Delta t / 2}^1,$
we have 
\begin{align*}
\mathcal F_{\Delta t}: 
\begin{pmatrix}
    \widetilde X^n\\
    \widetilde U^n\\
    \widetilde Y^n\\
    \widetilde V^n
\end{pmatrix}
\rightarrow
\begin{pmatrix}
\widetilde X^n+\sum\limits_{r=0}^m
\frac{\partial H_r(\hat  U^n,\hat Y^n)}{\partial Y}\Delta W_r^n\\
\hat U^n+\sum\limits_{r=0}^m
\frac{\partial H_r(\widetilde X^{n+1},\widetilde V^{n+1})}{\partial V}(W_r(t_{n+1})-W_r(t_n+\frac {\Delta t}2))\\
\hat Y^n-\sum\limits_{r=0}^m\frac{\partial H_r(\widetilde X^{n+1},\widetilde V^{n+1})}{\partial X}(W_r(t_{n+1})-W_r(t_n+\frac {\Delta t}2))\\
\widetilde V^n-\sum\limits_{r=0}^m
\frac{\partial H_r(\hat  U^n,\hat Y^n)}{\partial U}\Delta W_r^n
\end{pmatrix},
\end{align*}
where $n\in\{0,1,\ldots,N-1\},$ 
\begin{align*}
&\widetilde\xi^n=
    [(\widetilde X^n)^\top,
    (\widetilde U^n)^\top,
    (\widetilde Y^n)^\top,
    (\widetilde V^n)^\top]^\top,\\
    &\widetilde\xi^{n+1}=[(\widetilde X^{n+1})^\top,
    (\widetilde U^{n+1})^\top,
    (\widetilde Y^{n+1})^\top,
    (\widetilde V^{n+1})^\top]^\top,\\
 &   \hat Y^n=\widetilde Y^n -\sum\limits_{r=0}^m\frac{\partial H_r(\widetilde X^{n},\widetilde V^{n})}{\partial X}(W_r(t_n+\frac {\Delta t}2)-W_r(t_n)),\\
 &\hat U^n=\widetilde U^n -\sum\limits_{r=0}^m\frac{\partial H_r(\widetilde X^{n},\widetilde V^{n})}{\partial V}(W_r(t_n+\frac {\Delta t}2)-W_r(t_n)).
\end{align*}
\end{itemize}
\noindent 
It can be found that the numerical scheme given by $\widetilde{\xi}^{n+1}=\mathcal F_{\Delta t}(\widetilde{\xi}^{n})$ preserves the quadratic invariant $\hat Q_{\kappa}$ because of
$$
\frac 12(X^\top,V^\top)^\top\kappa_1
\begin{bmatrix}
\frac{\partial H_r(X,V)}{\partial X}\\
\frac{\partial H_r(X,V)}{\partial V}
\end{bmatrix}=0,\quad 
\frac 12(U^\top,Y^\top)^\top\kappa_1
\begin{bmatrix}
\frac{\partial H_r(U,Y)}{\partial U}\\
\frac{\partial H_r(U,Y)}{\partial Y}
\end{bmatrix}=0,$$
where $r=0,1,\ldots,m.$ 
As a result, $Q_\kappa\left(Z^{n+1}\right)=Q_\kappa\left(Z^{n}\right).$
\end{proof}

Now we perform numerical experiments to illustrate the validity of the proposed stochastic semi-explicit symplectic scheme for nonseparable SHSs by comparing with the  midpoint scheme and symplectic Euler scheme for \eqref{eq;SHS} with $m=1$ as follows
    \begin{equation}
    \label{midpoint}
\left\{
\begin{aligned}
X_{n+1}=X_{n}+ \sum\limits_{r=0}^1\frac{\partial H_r}{\partial Y}(\frac{X_{n+1}+X_{n}}2, \frac{Y_{n+1}+Y_{n}}2) \Delta W_r^n,\\
Y_{n+1}=Y_{n}-\sum\limits_{r=0}^1\frac{\partial H_r}{\partial X}(\frac{X_{n+1}+X_{n}}2, \frac{Y_{n+1}+Y_{n}}2) \Delta W_r^n,
\end{aligned}
\right.
\end{equation}
\begin{equation}
\label{SymEuler}
\left\{
\begin{aligned}
X_{n+1}=&X_{n}+ \sum\limits_{r=0}^1\frac{\partial H_r(X_{n+1}, Y_{n})}{\partial Y} \Delta W_r^n\\
&-\frac 12\big(\frac{\partial^2 H_1}{\partial Y^2}\frac{\partial H_1}{\partial X}
-
\frac 12\frac{\partial^2 H_1}{\partial Y\partial X}\frac{\partial H_1}{\partial Y}\big)(X_{n+1}, Y_{n}) \Delta t,\\
Y_{n+1}=&Y_{n}-\sum\limits_{r=0}^1\frac{\partial H_r(X_{n+1}, Y_{n})}{\partial X}\Delta W_r^n\\
&-\frac 12\big(\frac{\partial^2 H_1}{\partial X^2}\frac{\partial H_1}{\partial Y}-
\frac 12\frac{\partial^2 H_1}{\partial Y\partial X}\frac{\partial H_1}{\partial X}\big)(X_{n+1}, Y_{n}) \Delta t.   
\end{aligned}
\right.
\end{equation} 
For convenience, we denote semi-explicit symplectic schemes $\widetilde F_{\Delta t}$ based on \eqref{split1}, $\widetilde F_{\Delta t}$ based on \eqref{strang}, \eqref{midpoint} and \eqref{SymEuler} by SES-SP-1, SES-SP-2, Midpoint and Symplectic Euler, respectively.  
Recall that we need to solve the nonlinear equations
\begin{align*}
\varphi \left(\left(\Delta t, \xi^n\right),(\xi^{n+1}, \lambda)\right):=\left[\begin{array}{c}
\xi^{n+1}-\mathcal{F}_{\Delta t}\left(\xi^n+A^\top \lambda\right)-A^\top \lambda\\
A \xi^{n+1}
\end{array}\right].
\end{align*}
for $(\xi^{n+1}, \lambda),$ $n\in\{0,1,\ldots,N-1\}.$  
By taking product of $A$ in the first formula, one may eliminate $\xi^{n+1}$  and obtain $\mu$ satisfying 
$$A\mathcal{F}_{\Delta t}\left(\xi^n+A^\top \lambda\right)+AA^\top \lambda=0.$$ 
Its Jacobian matrix is
$$
D \phi_1\left(\left(\Delta t, \xi^n\right),(\xi^{n+1}, \lambda)\right)
:=A D \mathcal F_{\Delta t}\left(\xi^n+A^\top \mu\right) A^\top+A A^\top,
$$
and when $\Delta t=0$, $D \phi_1\left(\left(0, \xi^n\right),(\xi^{n+1}, \lambda)\right)=2A A^\top=4 I_{2 d}.$
Thus, we can exploit this simple structure of the Jacobian to construct the simplified Newton approximation
$$
\lambda^{(k+1)}=\lambda^{(k)}-\frac{1}{4} [A\mathcal{F}_{\Delta t}\left(\xi^n+A^\top \lambda^{(k)}\right)+AA^\top \lambda^{(k)}],
$$
where we start with $\lambda^{(0)}=0$. 
The simplified Newton iteration converges quickly with a reasonable tolerance $\epsilon$ imposed so they stop after $N$ iterations
$
\left\|\lambda^{(N+1)}-\lambda^{(N)}\right\|<\epsilon
$ and we set $\lambda=\lambda^{(N)}$. 
In all the numerical experiments, the expectation is approximated by taking the average over 1000 realizations, and the exact solution is computed by implementing the corresponding proposed numerical schemes with fine mesh.  

\begin{ex}
\begin{figure}[h!]
	\centering
	\subfigure{
		\begin{minipage}{12.5cm}
\centering
\includegraphics[height=3.5cm,width=4cm]{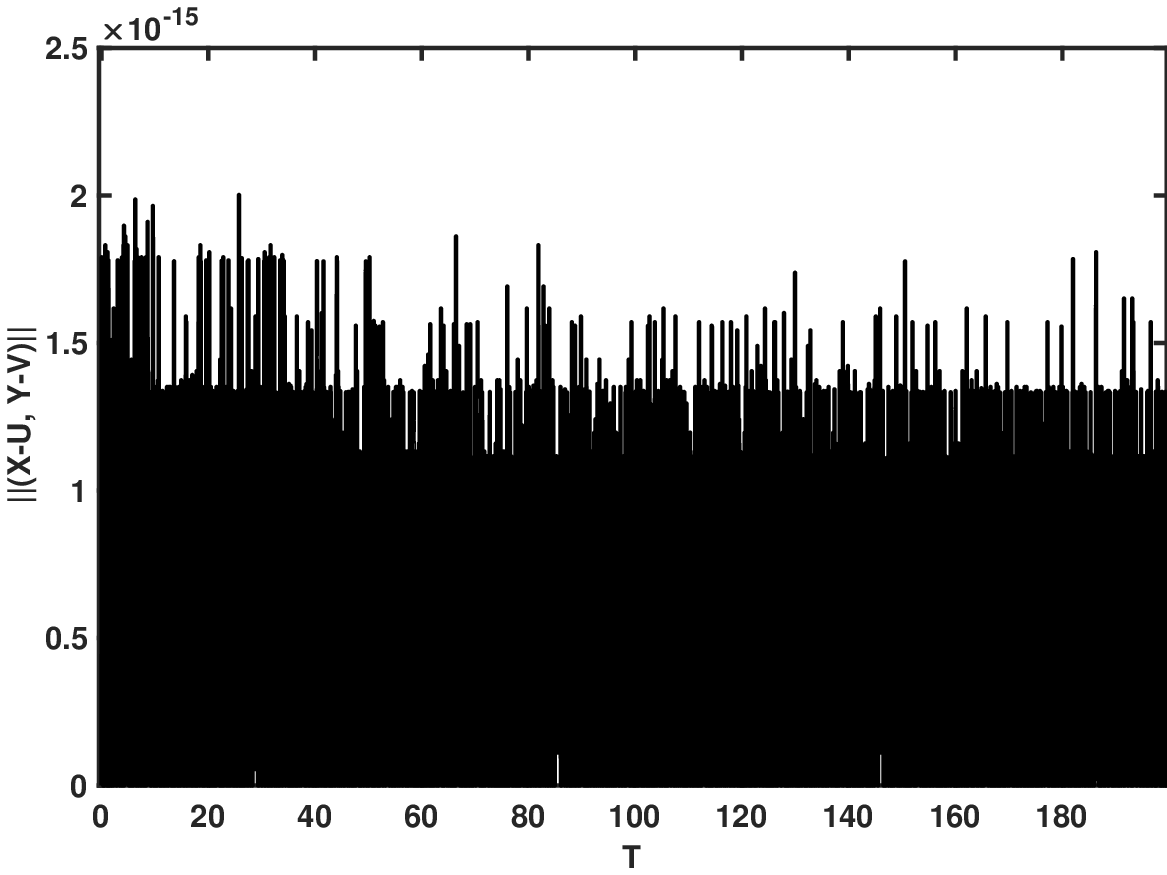}
\includegraphics[height=3.5cm,width=4cm]{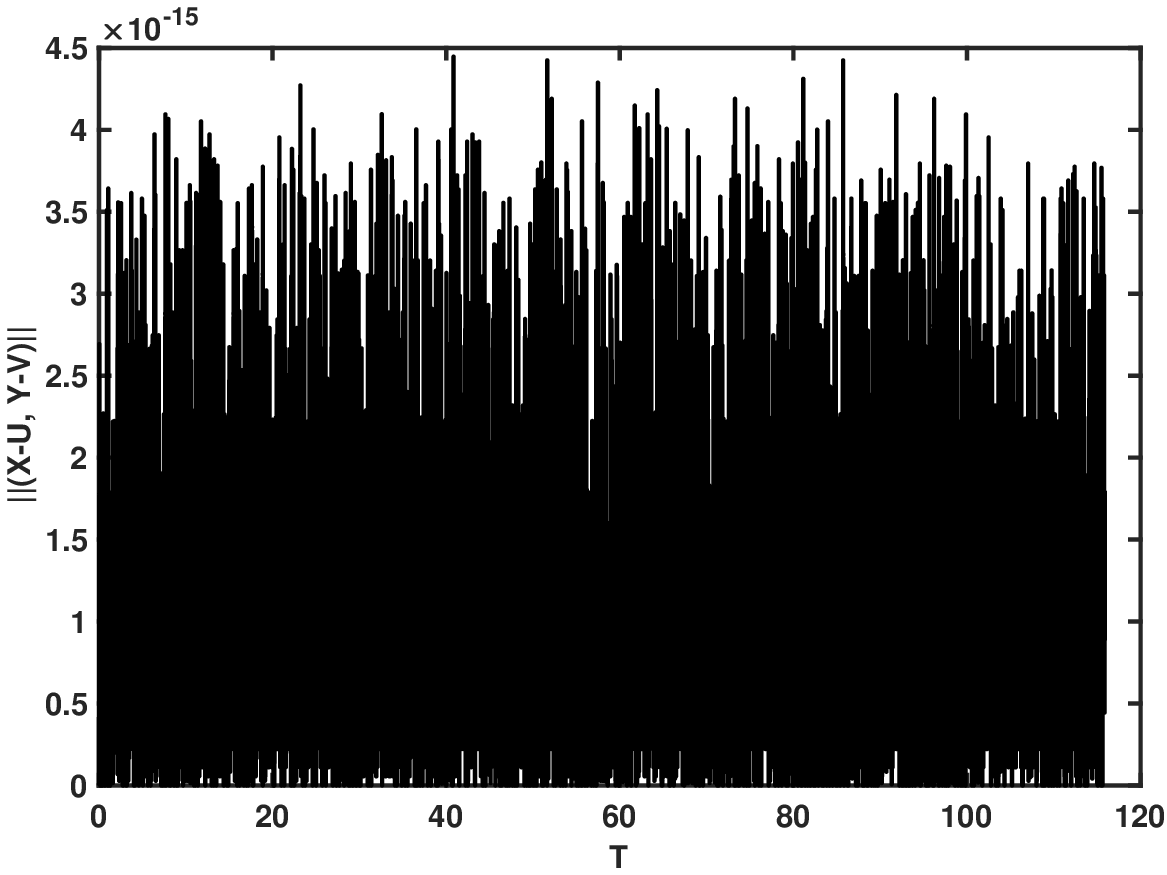}
\end{minipage}
	}
	\caption{Time evolutions of norm $\|(X-U, Y-V)\|$ with 
 $\Delta t = 1\times 10^{-2}, c = 0.5, \gamma = 1 $: (left)SES-SP-1, (right) SES-SP-2.} 
	\label{fig9}
\end{figure}
\begin{figure}[h!]
	\centering
	\subfigure{
		\begin{minipage}{12.5cm}
\centering
\includegraphics[height=3.5cm,width=4cm]{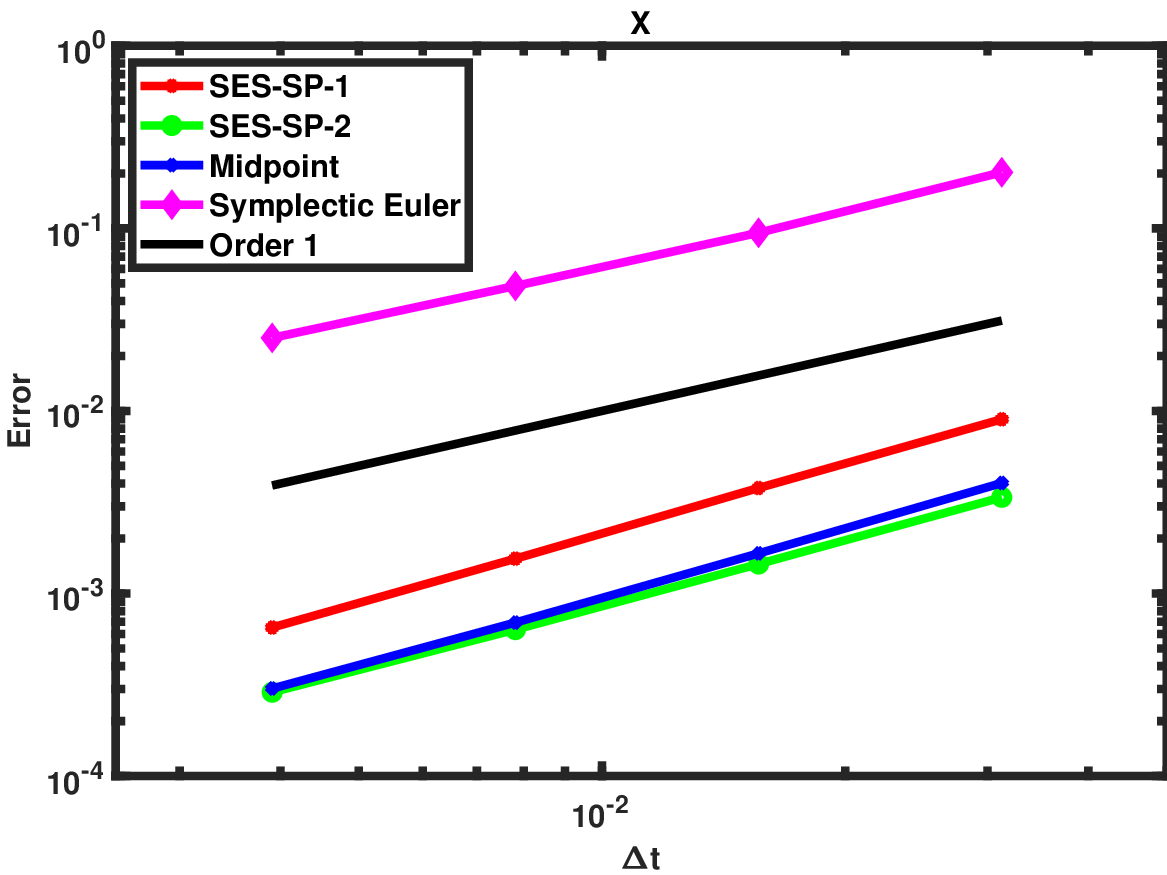}
\includegraphics[height=3.5cm,width=4cm]{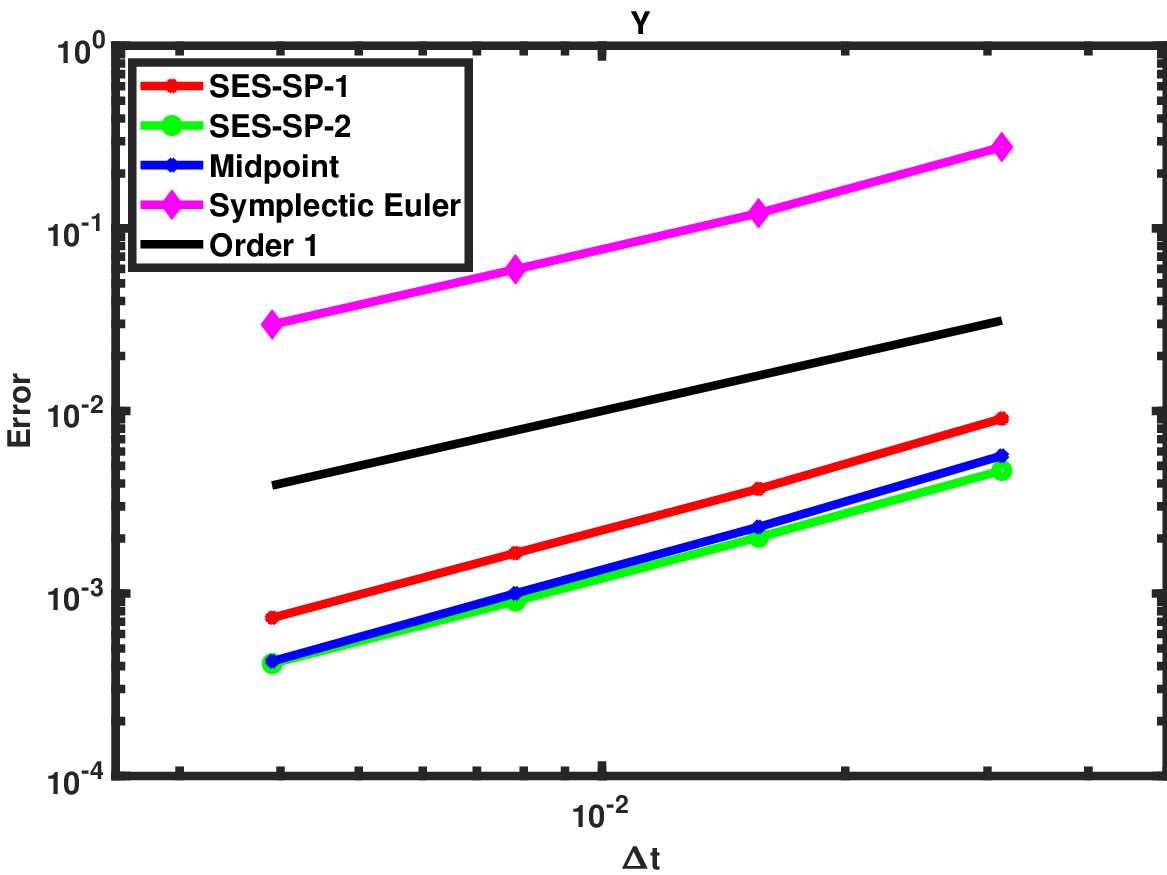}
\end{minipage}
	}
	\caption{Mean-square convergence order of various schemes in time with $T=1, c = 0.15, \gamma = 0.01$.} 
	\label{fig7}
\end{figure} 

\begin{table}[htbp]  
\setlength{\abovecaptionskip}{0.1cm}
	\setlength{\belowcaptionskip}{-0.cm}
	\caption{Numerical errors and CPU time of numerical schemes for Example 1 with 
$c=0.4,\gamma = 0.5$.}
	\centering
	\begin{tabularx}{\textwidth}{p{0.9cm}<{\centering}|p{1.7cm}p{1.7cm}<{\centering}p{1.7cm}<{\centering}|p{1.5cm}<{\centering}p{1.5cm}<{\centering}p{1.5cm}<{\centering}}
		\toprule
		 & \multicolumn{1}{c}{}  & \multicolumn{2}{c|}{Errors} & \multicolumn{3}{
c}{CPU time} \\
		\toprule
		$\Delta t$ & SES-SP-1 &SES-SP-2 & Midpoint &SES-SP-1 &SES-SP-2 & Midpoint  \\
		\midrule
		$1/2^6$     &  5.0315e-02  & 2.5054e-02 &2.4028e-02& 4.38s & 6.37s &  
10.41s  \\
		$1/2^8$     &  1.1440e-02  & 5.5905e-03 &5.6065e-03& 16.24s & 22.50s& 38.68s  \\
		$1/2^{10}$ &  2.6952e-03  & 1.3353e-03 &1.4048e-03& 60.62s & 82.83s & 141.23s \\
		$1/2^{12}$ &  5.8261e-04  & 3.0282e-04 &3.4664e-04& 248.72s & 324.61s & 554.06s \\
		\bottomrule
	\end{tabularx}
	\label{table1}
\end{table}

\begin{figure}
	\centering
	\subfigure{
		\begin{minipage}{12.5cm}
\centering
\includegraphics[height=3.5cm,width=4cm]{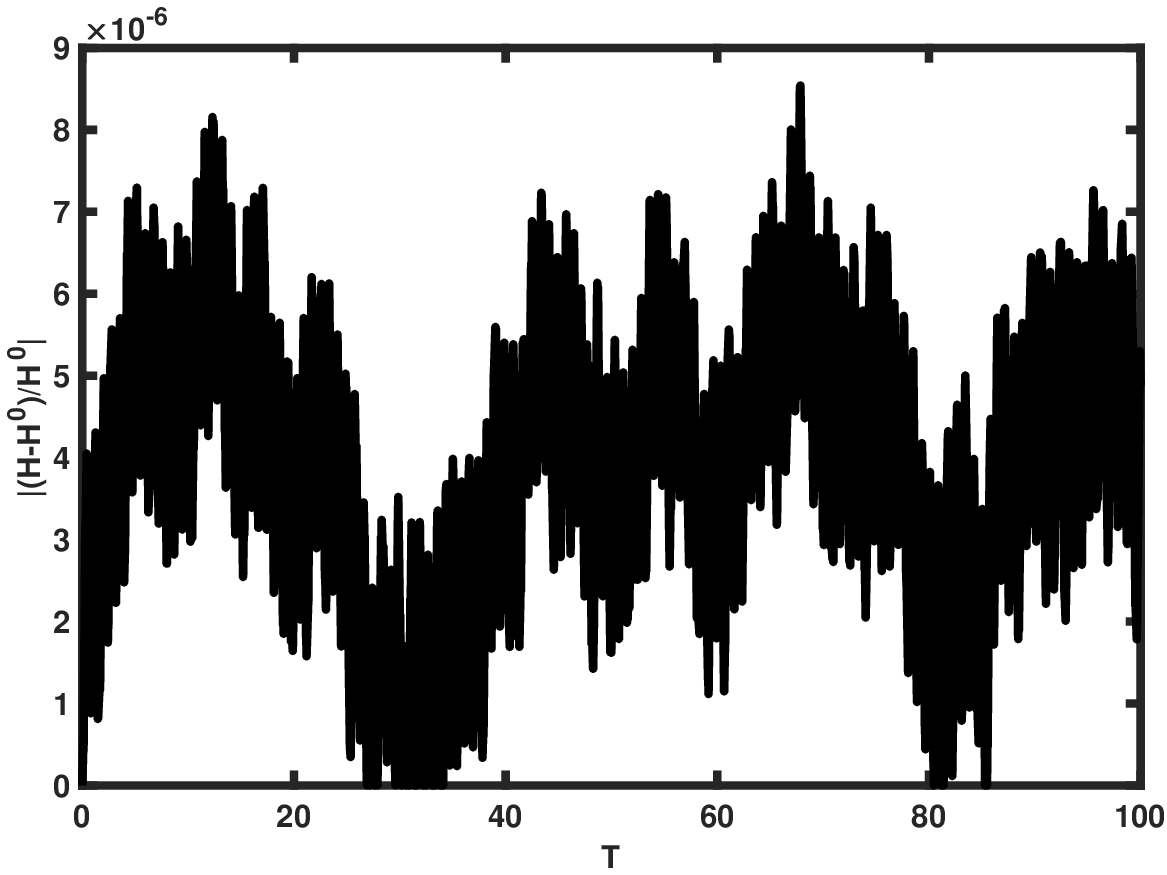}
\includegraphics[height=3.5cm,width=4cm]{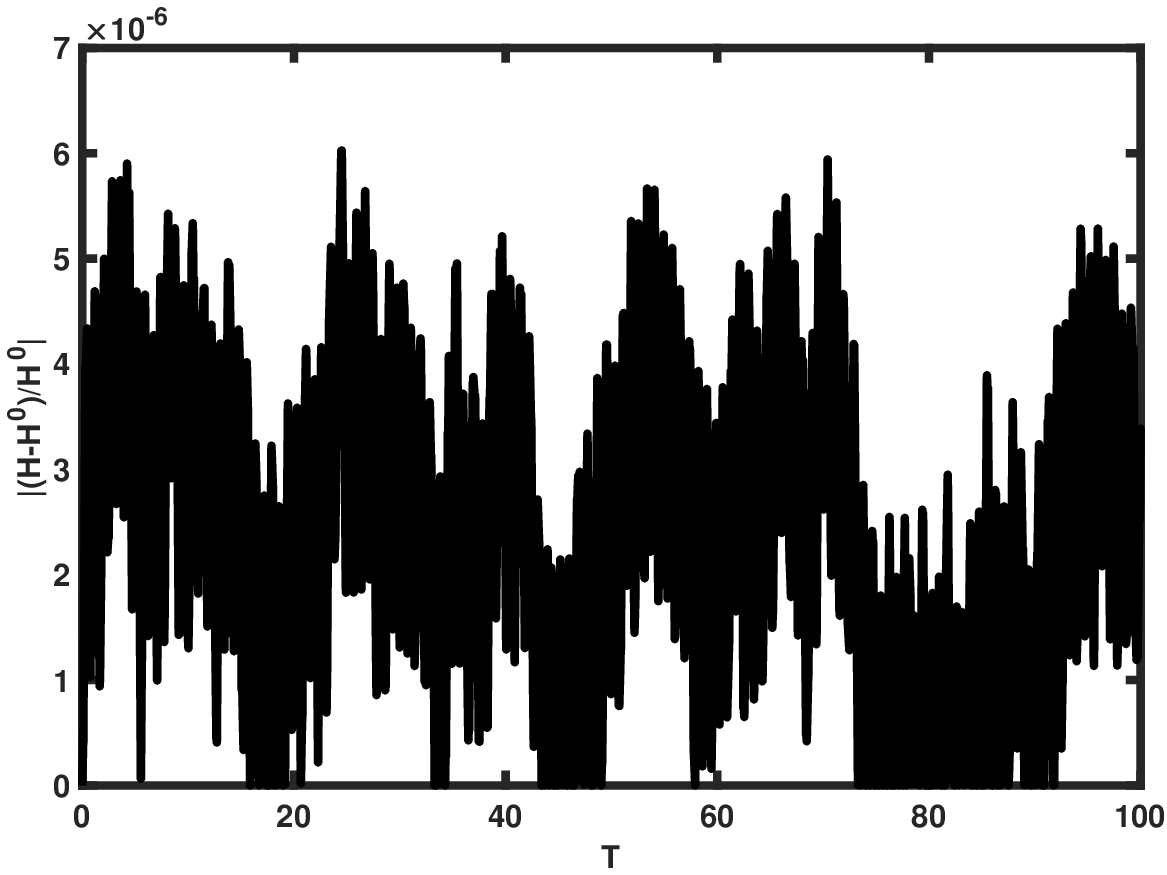}\\
\includegraphics[height=3.5cm,width=4cm]{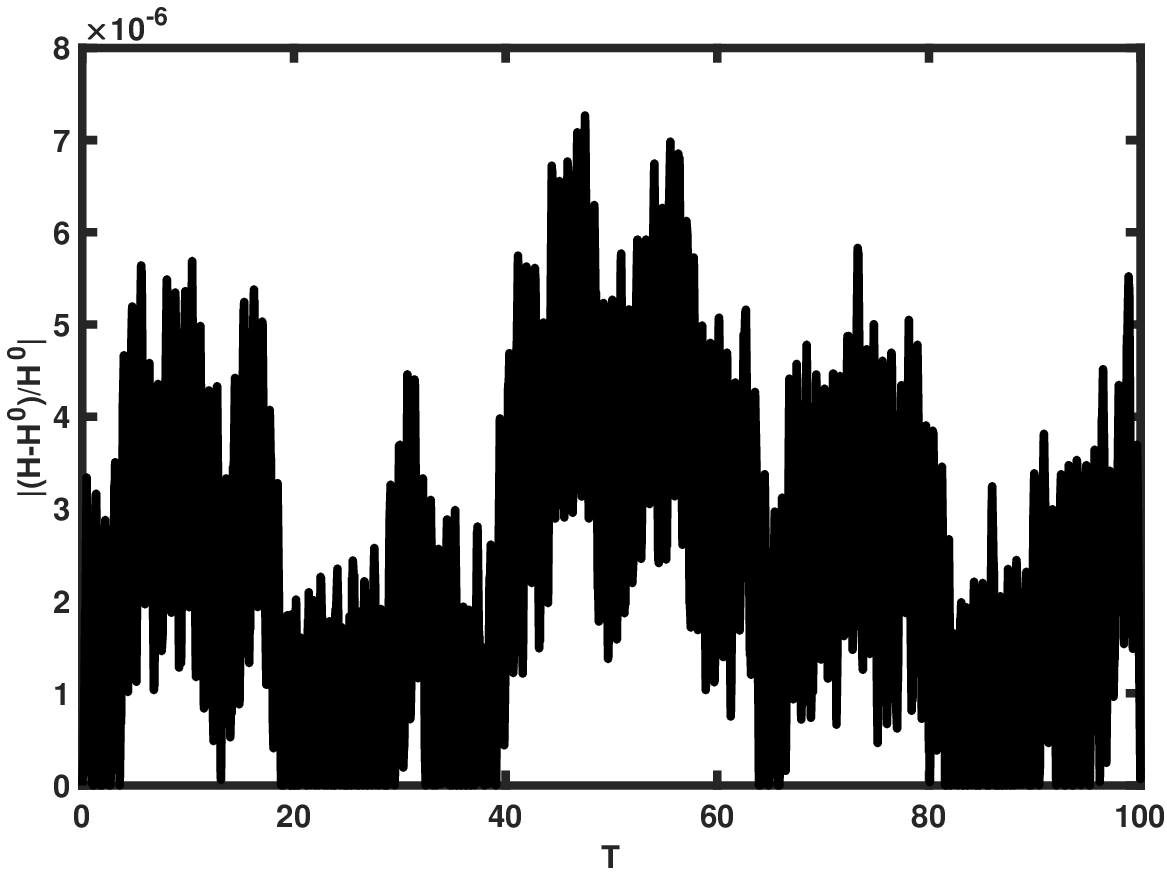}
\includegraphics[height=3.5cm,width=4cm]{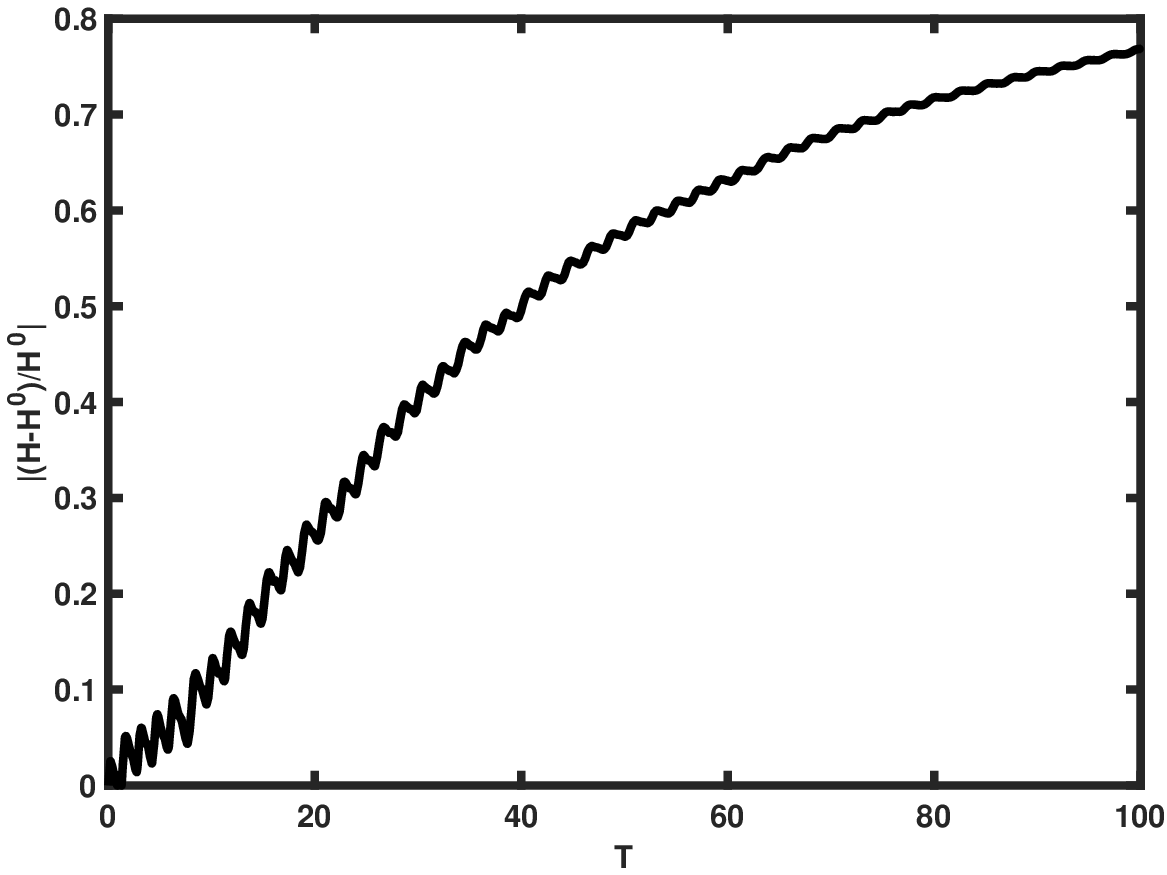}
\end{minipage}
	}
	\caption{Relative Hamiltonian $H_0$ of various schemes: with $T =100, \Delta t = 1\times 10^{-4}, c = 0.1, \gamma = 0 $: (1)SES-SP-1, (2) SES-SP-2, (3)Midpoint, (4)Symplectic Euler.} 
	\label{fig8}
\end{figure}

Consider 
\begin{equation}
\label{eq;Ex1}
\left\{\begin{aligned}
&d X= (X^2+1)Ydt + c(X^2+1)Y\circ d W,\quad X(0) = 0,\\
&d Y= -(Y^2+1)Xdt - c(Y^2+1)X\circ d W,\quad Y(0) = -3,
\end{aligned}\right.
\end{equation}
where Hamlitonians $H_0(X,Y)=\frac{1}{2}(X^2+1)(Y^2+1)$ and $H_1=cH_0$ with $c\in\mathbb R\setminus \{0\}$ are invariants. 
Fig. \ref{fig9} illustrates that the defect $(\widetilde X^n-\widetilde U^n, \widetilde Y^n-\widetilde V^n)$ of semi-explicit symplectic schemes $\widetilde F_{\Delta t}$ based on \eqref{split1} and \eqref{strang} is very small, confirming the feasibility of the proposed scheme. In Fig. \ref{fig7}, the mean-square convergence order of errors $\sqrt{\mathbb E(\|X(T)-X_N\|^2}$ and $\sqrt{\mathbb E(\|Y(T)-Y_N\|^2)}$ against $\Delta t$ with $T = 1$ is displayed on a log-log scale with $\Delta t=2^{-s}, s\in\{5, 6, 7, 8\}$. The slopes of the four stochastic symplectic schemes closely align with a reference line (black color). 

Table \ref{table1} presents the efficiency and accuracy comparisons of stochastic semi-explicit symplectic schemes and stochastic midpoint scheme under the same time stepsize $\Delta t=2^{-s}, s\in\{6, 8, 10, 12\}$.
It is evident that the stochastic semi-explicit symplectic scheme $\widetilde F_{\Delta t}$ based on \eqref{split1} exhibits higher efficiency, while the semi-explicit symplectic scheme $\widetilde F_{\Delta t}$ based on \eqref{strang} and the stochastic midpoint scheme are more accurate.
Furthermore, the error of $\widetilde F_{\Delta t}$ based on \eqref{strang} is comparable to that of the midpoint scheme, but the computational time of the midpoint scheme is approximately 1.5 times longer than that of the former. 
Additionally, Fig. \ref{fig8} demonstrates that stochastic semi-explicit symplectic schemes and stochastic midpoint scheme preserve the relative Hamiltonian. However, it is worth noting that as time progresses toward infinity, the relative Hamiltonian of the symplectic Euler scheme increases, indicating that the stochastic symplectic Euler scheme is unable to preserve the Hamiltonian $H_0.$ 
\end{ex}

\begin{ex}
\begin{figure}
	\centering
	\subfigure{
		\begin{minipage}{12.5cm}
\centering
\includegraphics[height=3.5cm,width=4cm]{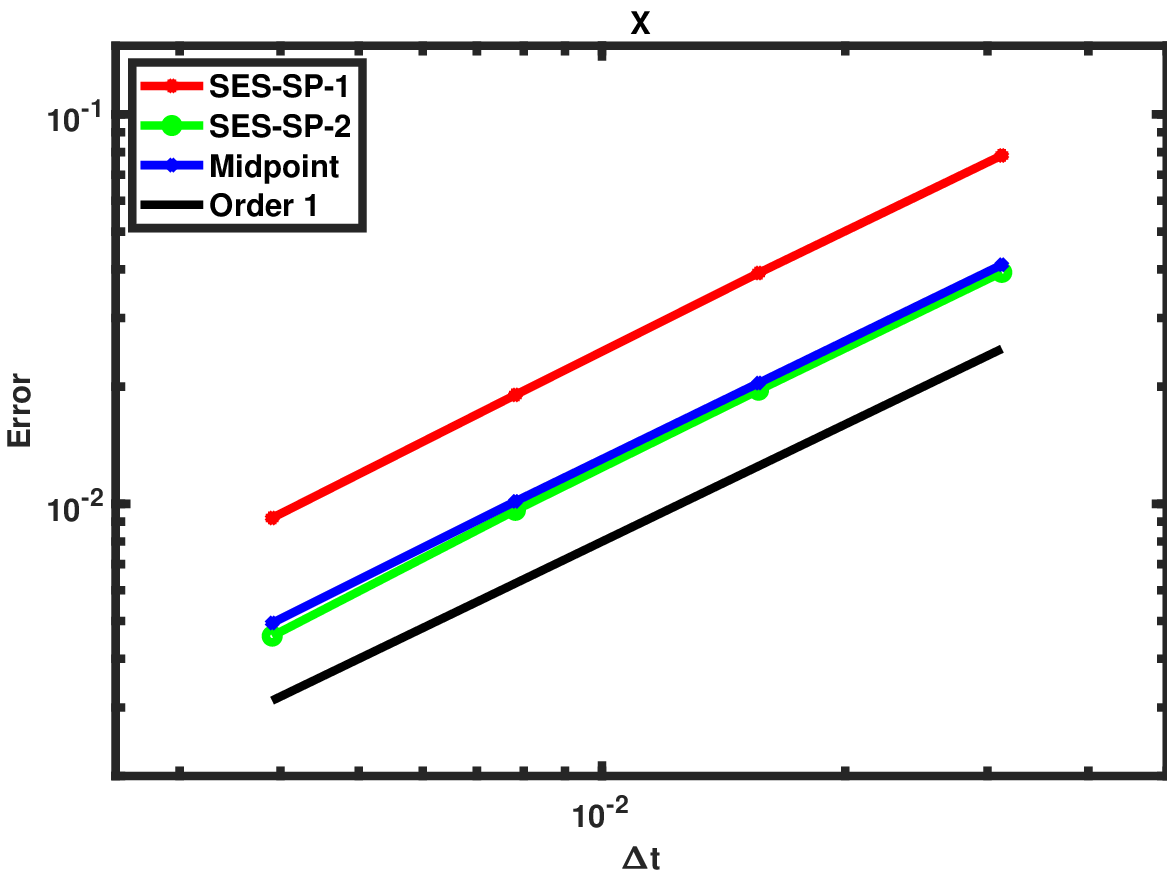}
\includegraphics[height=3.5cm,width=4cm]{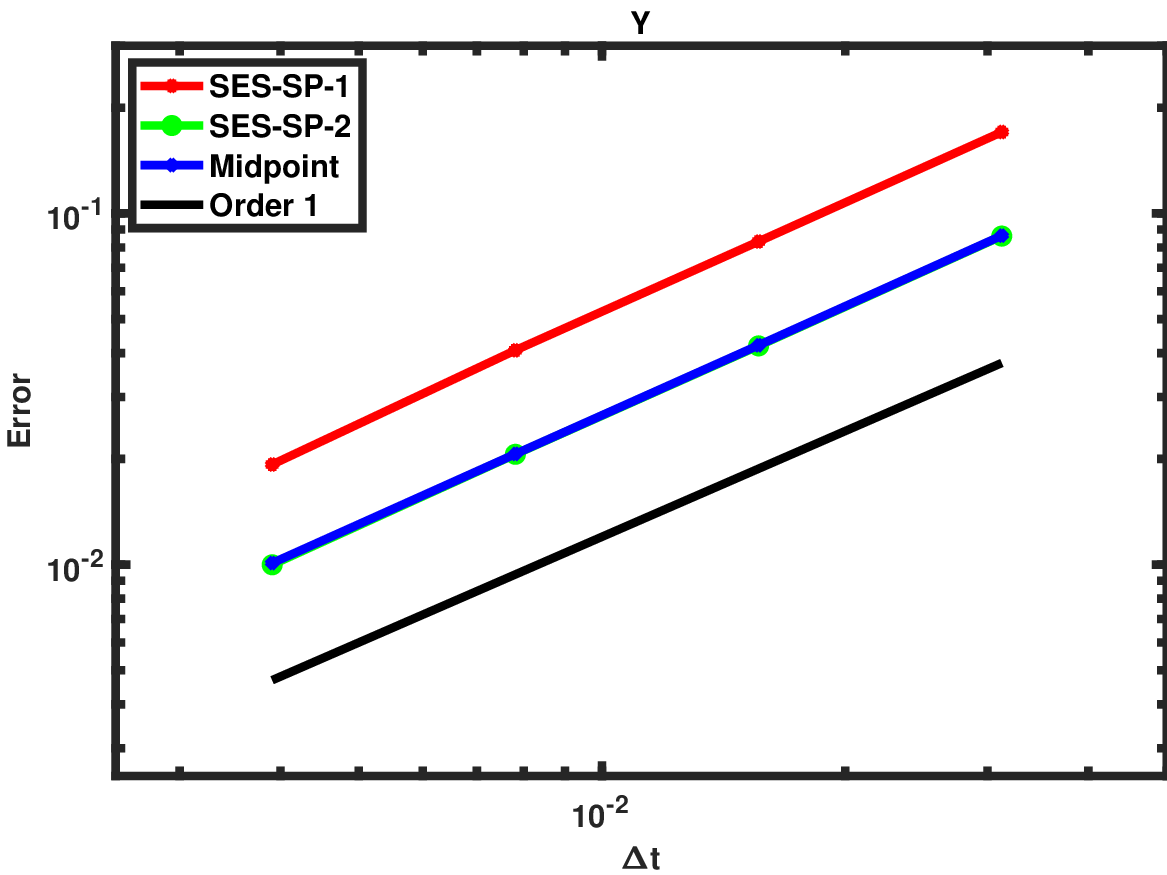}
\end{minipage}
	}
	\caption{Mean-square convergence order of various schemes with $T=1, c = 0.5, \gamma = 0.5$ and $\Delta t=2^{-s}, s\in\{5, 6, 7, 8\}$.} 
	\label{fig1}
\end{figure}

\begin{figure}
	\centering
	\subfigure{
		\begin{minipage}{12.5cm}
\centering
\includegraphics[height=3.5cm,width=4cm]{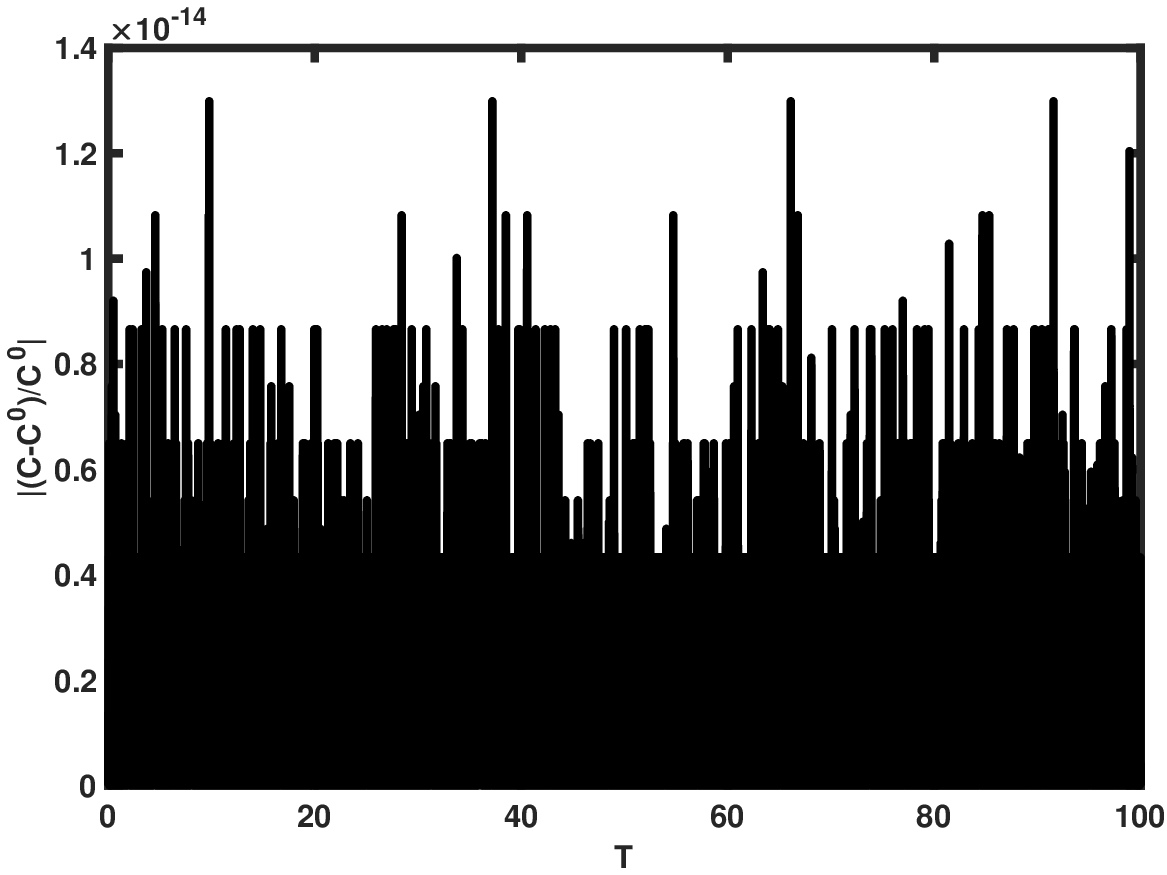}
\includegraphics[height=3.5cm,width=4cm]{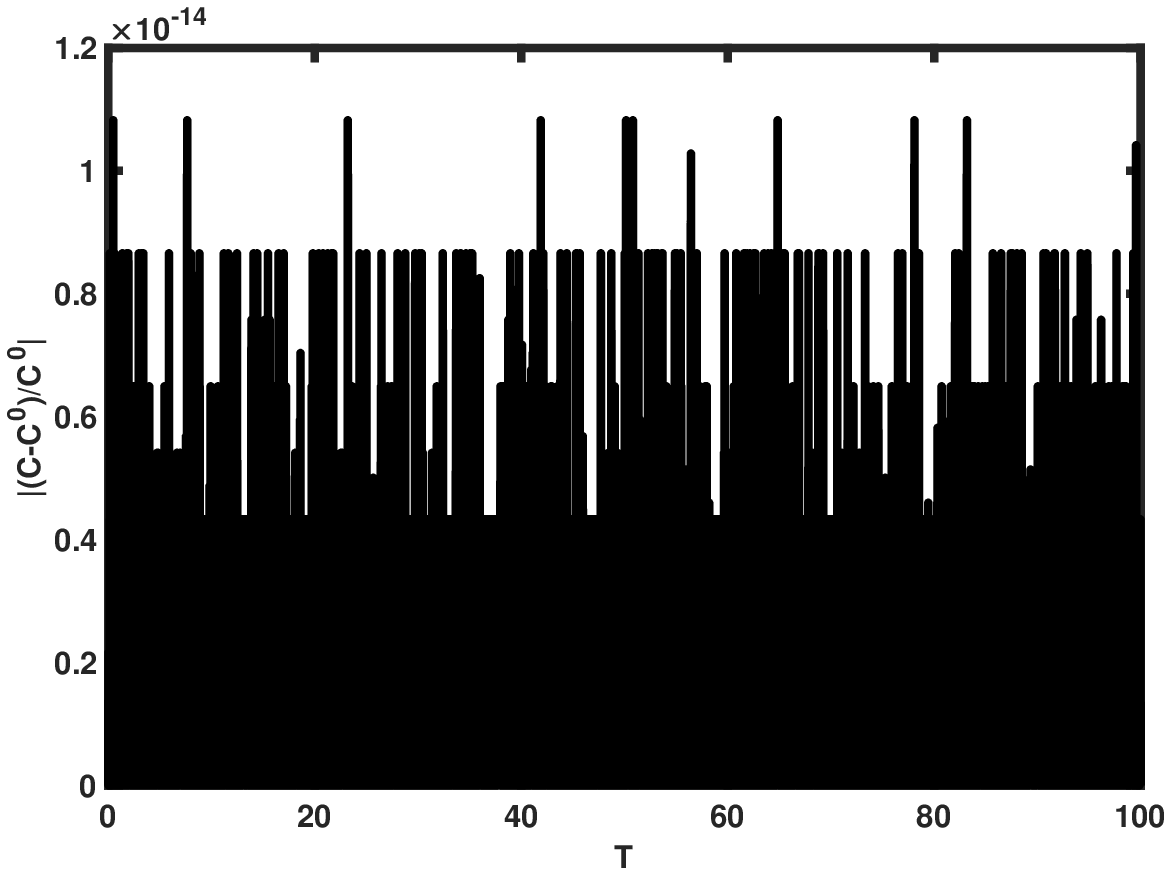}
\includegraphics[height=3.5cm,width=4cm]{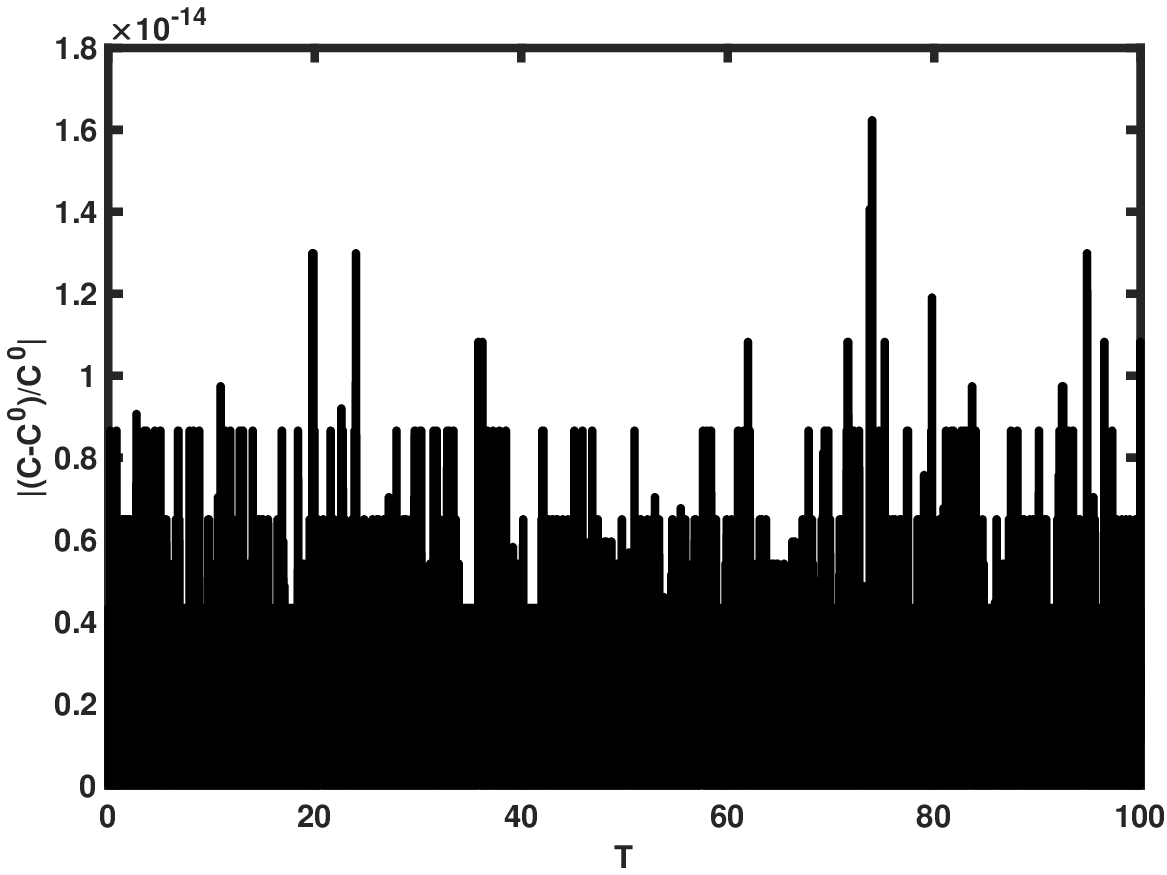}
\end{minipage}
	}
	\caption{Relative Casimir function with $T =100, \Delta t = 1.25\times 10^{-2}, c = 1, \gamma = 1 $: (1)SES-SP-1, (2) SES-SP-2, (3)Midpoint.} 
	\label{fig2}
\end{figure}

\begin{table}[htbp]  
	\setlength{\abovecaptionskip}{0.1cm}
	\setlength{\belowcaptionskip}{-0.cm}
	\caption{Numerical errors and CPU time of numerical schemes for Example 2 with $c=0.2, \gamma = 2$.}
	\centering
	\begin{tabularx}{\textwidth}{p{0.9cm}<{\centering}|p{1.7cm}p{1.7cm}<{\centering}p{1.7cm}<{\centering}|p{1.5cm}<{\centering}p{1.5cm}<{\centering}p{1.5cm}<{\centering}}
		\toprule
		 & \multicolumn{1}{c}{}  & \multicolumn{2}{c|}{Errors} & \multicolumn{3}{c}{CPU time} \\
		\toprule
		$\Delta t$ & SES-SP-1 &SES-SP-2 & Midpoint &SES-SP-1 &SES-SP-2 & Midpoint  \\
		\midrule
		$1/2^7$     &  7.6965e-03  & 3.8460e-03 &4.0291e-03& 10.94s & 12.41s &  18.29s  \\
		$1/2^9$     &  1.7873e-03  & 9.0099e-04 &9.0980e-04& 39.37s & 45.46s& 62.83s  \\
		$1/2^{13}$ &  1.1155e-04  & 5.6244e-05 &5.7225e-05& 605.72s & 625.89s &974.68s \\
		$1/2^{14}$ &  5.0833e-05  & 2.6950e-05 &2.7953e-05& 1201.37s & 1245.18s & 1878.58s \\
		\bottomrule
	\end{tabularx}
		\label{table2}
\end{table}
Consider the 2-dimensional nonseparable SHS
\begin{equation}
\left\{
\begin{aligned}
&d\left[\begin{array}{l}
X \\
Y
\end{array}\right]=\left[\begin{array}{cc}
0 & -1 \\
1 & 0
\end{array}\right] \nabla H(X, Y)(dt + c \circ \mathrm{dW}(t)),\\
&X(0)=\ln y_3^0,\quad  Y(0)=\ln y_2^0,
\end{aligned}
\right.
\end{equation}
where Hamiltonian 
$$H(X, Y)=ab \exp(v(X-\mathcal{C}+b Y))+\exp(-Y)-\omega Y -a\exp(X)- \mu X$$
with $\mathcal{C}=-\frac{1}{v} \ln y^0_1-b \ln y^0_2+\ln y^0_3,$  $a,b,\mu,\omega\in\mathbb R\backslash \{0\}$ and $y_1^0,y_2^0,y_3^0>0.$ 
Let $
y_1 = \exp(v(X-\mathcal{C}+bY)),$ $y_2 =\exp(-Y),$ $y_3=\exp(X).$ 
We have the three-dimensional Lotka-Volterra system as follows
\begin{equation}
\label{eq;LV system}
\left\{
\begin{aligned}
&d\left[\begin{array}{l}
y_1\\
y_2\\
y_3
\end{array}\right]
=
\begin{bmatrix}
0 & vy_1y_2 & bvy_1y_3 \\
-vy_1y_2 & 0 & -y_2y_3 \\
-bvy_1y_3 & y_2y_3 & 0 
\end{bmatrix}
\nabla \tilde H(y_1,y_2,y_3
)(d t+c \circ d W(t)),\\
&(y_1(0)=y^0_1,\, y_2(0)=y^0_2,\, y_3(0)=y^0_3,
\end{aligned}
\right.
\end{equation}
with
$\widetilde H(\mathbf{y})=a b y_1+y_2+\omega\ln y_2-a y_3-\mu \ln y_3.$ 
The three-dimensional Lotka--Volterra system 
describes the interactions between three populations or factors and hold an invariant $\mathbf{C}(y_1,y_2,y_3)=-\frac{1}{v}\ln y_1-b \ln y_2+\ln y_3 =\mathcal C$, which is called Casimir function. 

In the numerical experiment, we set $a=-2, b=-1$, $v=-0.5, \omega=1, \mu=2$, and initial conditions $y_1^0=1, y_2^0=1.9, y_3^0=0.5.$  
Compared with the reference line in Fig. \ref{fig1}, it can be observed that the mean-square convergence order of semi-explicit symplectic  schemes and midpoint scheme is 1. 
Fig. \ref{fig2} shows that the numerical scheme based on the nonlinear mapping and stochastic semi-explicit symplectic schemes 
$\widetilde F_{\Delta t}$ with \eqref{split1} or \eqref{strang}  for the stochastic Lotka-Volterra system \eqref {eq;LV system} can more accurately preserve the Casimir function  compared to the midpoint scheme, because the relative Casimir function error for the midpoint scheme sometimes is larger. 

Table \ref{table2} shows that the semi-implicit symplectic scheme $\widetilde F_{\Delta t}$  derived from \eqref{strang} is the most accurate, whereas the semi-implicit symplectic scheme $\widetilde F_{\Delta t}$  derived from \eqref{split1} is the most efficient. 
Furthermore, the error of $\widetilde F_{\Delta t}$ based on \eqref{strang} is similar to that of the midpoint scheme.  
However, under the same level of accuracy, the semi-implicit symplectic scheme  $\widetilde F_{\Delta t}$ based on \eqref{strang} takes less time.

\end{ex}

\begin{ex}
\begin{figure}[h!]
	\centering
	\subfigure{
		\begin{minipage}{12.5cm}
\centering
\includegraphics[height=3.5cm,width=4cm]{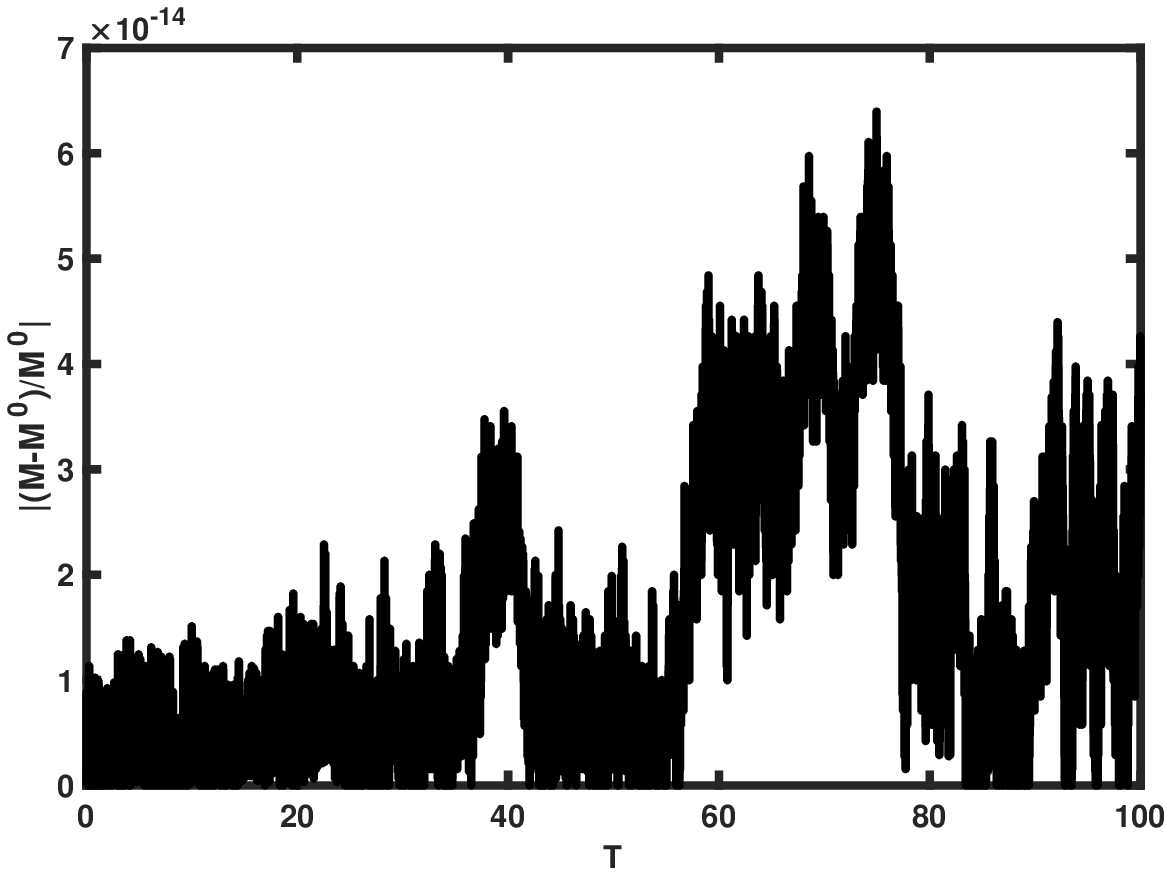}
\includegraphics[height=3.5cm,width=4cm]{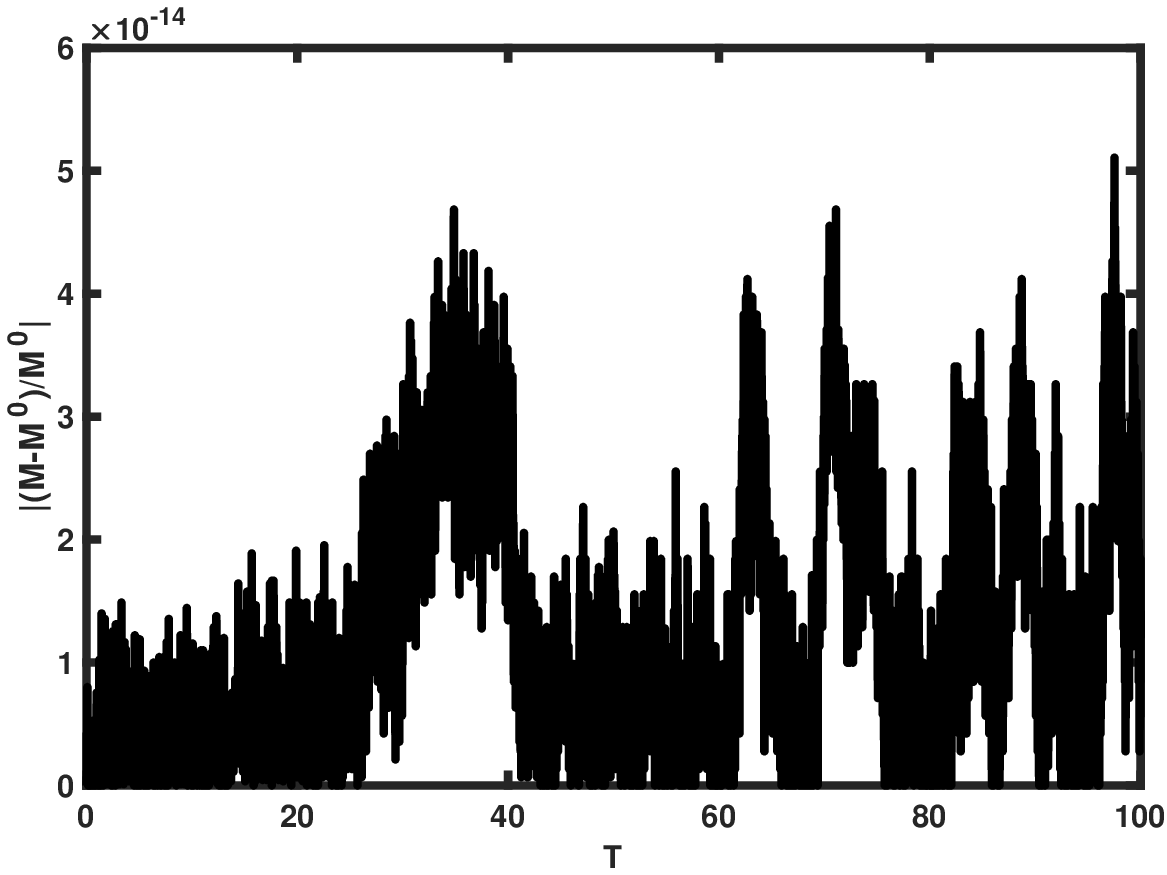}
\end{minipage}
	}
	\caption{Relative linear invariant with $T =100, \Delta t = 1\times 10^{-2}, c = 0.5, \gamma = 0.5 $: (1)SES-SP-1, (2) SES-SP-2.
 } 
	\label{fig11}
\end{figure}

\begin{figure}[h!]
	\centering
	\subfigure{
		\begin{minipage}{12.5cm}
\centering
\includegraphics[height=3.5cm,width=4cm]{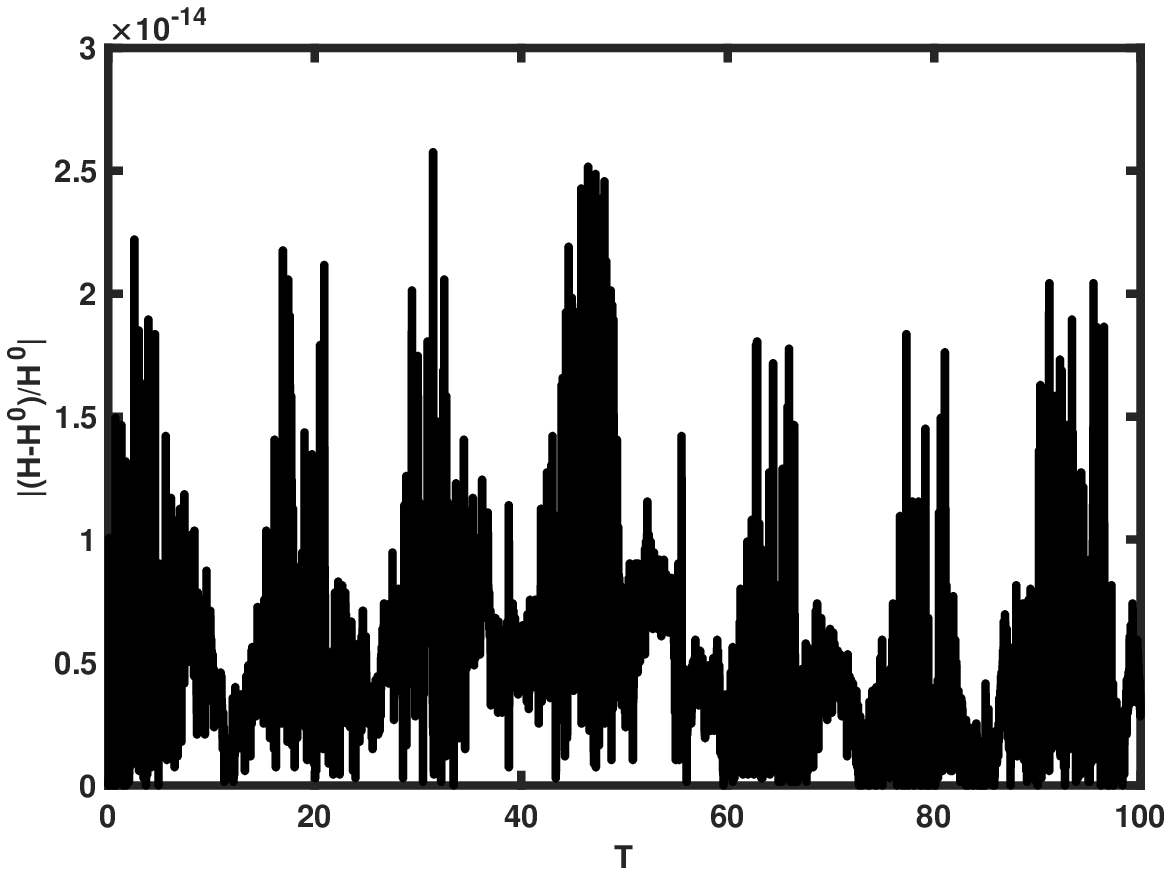}
\includegraphics[height=3.5cm,width=4cm]{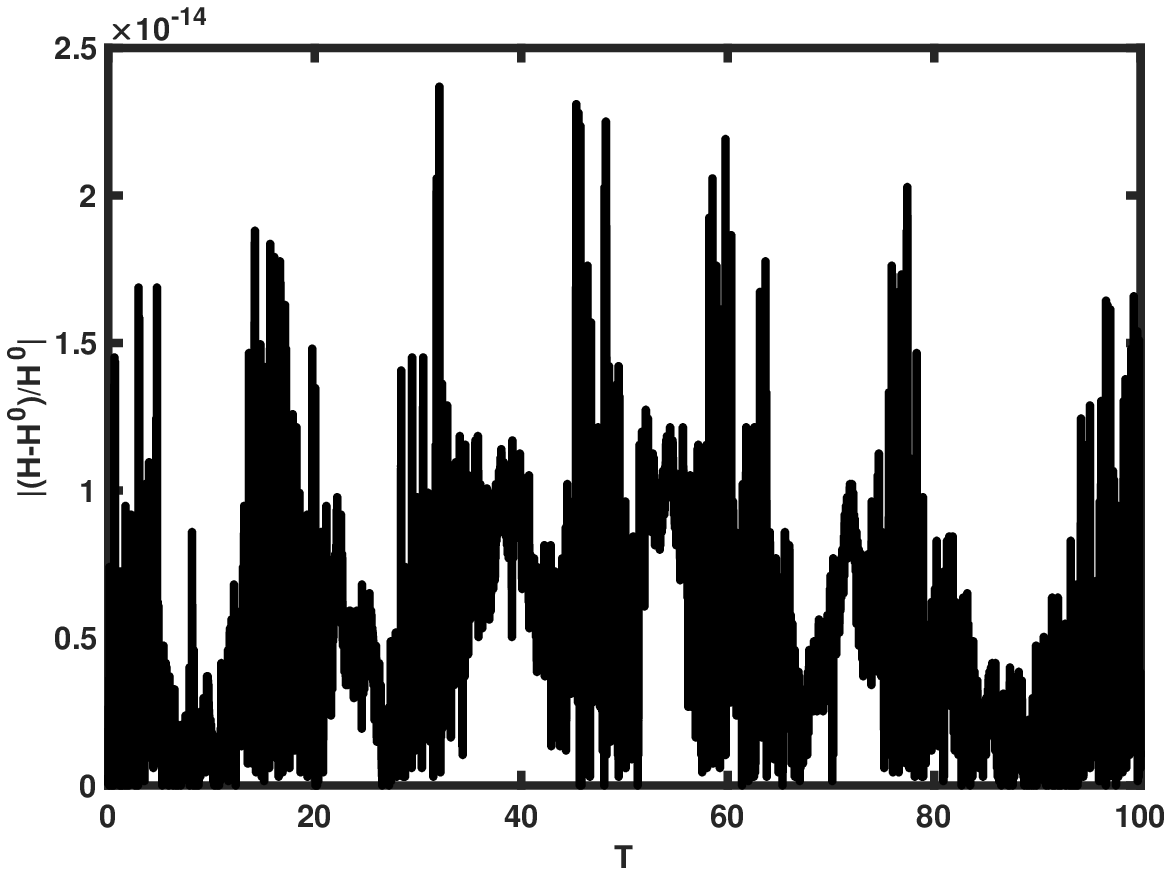}
\end{minipage}
	}
	\caption{Relative quadratic invariant with $T =100, \Delta t = 1\times 10^{-2}, c = 0.5, \gamma = 0 $: (1)SES-SP-1, (2) SES-SP-2.
 } 
	\label{fig12}
\end{figure} 
Consider a $4d$-dimensional nonseparable SHS, 
where $$H_0(X,Y) = \exp(f(X_1, Y_1)sin(g(X_2, Y_2))), \, H_1(X,Y) = cH_0(X,Y)$$ with $X=(X_1,X_2)^\top, Y=(Y_1,Y_2)^\top,$ $f(X_1, Y_1) = \frac{1}{10}(2X_1-3Y_1)$ and $g(X_2, Y_2) = \frac{1}{4}(X_2^2 + 2Y_2^2)$. 
It can be verified that  
$f(X, Y)$ and $g(X, Y)$ are linear invariant and quadratic invariant of SHS, respectively. 

Set the initial value $ X(0) = (-1,2)^\top, Y(0) = (1,-1)^\top.$
From Figs. \ref{fig11} and \ref{fig12}, it is obvious that the semi-implicit symplectic scheme $\widetilde F_{\Delta t}$ derived from \eqref{split1} and \eqref{strang} preserves both the linear and quadratic invariants, aligning with theoretical results. 
Table \ref{table3} indicates that the efficiency of the midpoint scheme is not as good as that of the semi-explicit symplectic scheme $\widetilde F_{\Delta t}$ derived from \eqref{split1} and \eqref{strang}
, while its accuracy is comparable to that of the semi-explicit symplectic scheme $\widetilde F_{\Delta t}$ based on \eqref{strang}. 

\begin{table}[htbp]  
\setlength{\abovecaptionskip}{0.1cm}
	\setlength{\belowcaptionskip}{-0.cm}
	\caption{Numerical errors and CPU time of numerical schemes for Example 3 with $c=0.5, \gamma = 1$.}
	\centering
	\begin{tabularx}{\textwidth}{p{0.9cm}<{\centering}|p{1.7cm}p{1.7cm}<{\centering}p{1.7cm}<{\centering}|p{1.5cm}<{\centering}p{1.5cm}<{\centering}p{1.5cm}<{\centering}}
		\toprule
		 & \multicolumn{1}{c}{}  & \multicolumn{2}{c|}{Errors} & \multicolumn{3}{c}{CPU time} \\
		\toprule
		$\Delta t$ & SES-SP-1 &SES-SP-2 & Midpoint &SES-SP-1 &SES-SP-2 & Midpoint  \\
		\midrule
		$1/2^5$     &  4.1817e-03  & 1.8663e-03 &1.8724e-03& 56.72s & 63.95s &  70.83s  \\
		$1/2^7$     &  9.3747e-04  & 4.5617e-04 &4.6078e-04& 185.27s & 221.14s& 254.81s  \\
		$1/2^9$     &  2.3144e-04  & 1.1361e-04 &1.1640e-04& 649.84s & 814.78s & 978.17s \\
		$1/2^{10}$  &  1.1685e-04  & 5.4829e-05 &5.5413e-05& 1188.46s & 1498.51s &1887.43s \\
		\bottomrule
	\end{tabularx}
	\label{table3}
\end{table}

\end{ex}

\begin{ex}
Take into account the nonseparable SHS as follows
\begin{equation}
\left\{
\begin{aligned}
& d\left[\begin{array}{l}
X \\
Y
\end{array}\right]=\left[\begin{array}{cc}
0 & -1 \\
1 & 0
\end{array}\right] \nabla H\left(X, Y\right)\left(d t + c \circ d W(t)\right),
\\
&X=y_2^0, \quad Y=\arctan \left(\frac{y_3^0}{y_1^0}\right),
\end{aligned}
\right.
 \end{equation}
where $ 
H\left(X, Y\right)=\frac{1}{2 I_1}\left(2 \mathcal{C}_1-X^2\right) \cos ^2\left(Y\right)+\frac{1}{2 I_2} X^2+\frac{1}{2 I_3}\left(2 \mathcal{C}_1-X^2\right) \sin ^2\left(Y\right)$ 
with 
$I_1 = \sqrt{2}+ \sqrt{\frac{2}{1.51}}, I_2 = \sqrt{2}-0.51\sqrt{ \frac{2}{1.51}}, I_3 =1,$ 
$c\in\mathbb R$ and $\mathcal{C}_1=\frac{1}{2}\left(\left(y_1^0\right)^2+\left(y_2^0\right)^2+\left(y_3^0\right)^2\right)$.  
Setting \begin{equation*}
y_1=\sqrt{2 \mathcal{C}_1-X^2} \cos \left(Y\right), \quad y_2=X, \quad y_3=\sqrt{2 \mathcal{C}_1-X^2} \sin \left(Y\right), 
 \end{equation*}
we obtain the stochastic rigid body system  
\begin{equation}\label{eq;rigid body}
d \mbf{y} = S(\mbf{y}) \nabla K(\mbf{y})\left(d t + c \circ d W(t)\right), \quad \mbf{y}(0)=\left(y_1^0, y_2^0, y_3^0\right)^{\top},
 \end{equation}
where $K=\frac{1}{2}\left(\frac{y_1^2}{I_1}+\frac{y_2^2}{I_2}+\frac{y_3^2}{I_3}\right)$ and $S =\left[\begin{array}{ccc}
0 & -y_3 & y_2 \\
y_3 & 0 & -y_1 \\
-y_2 & y_1 & 0
\end{array}\right].$ 
It possesses Casimir function 
$
C(\mbf{y})=\frac{1}{2}\left(y_1^2+y_2^2+y_3^2\right)=: \mathcal{C}_1.$

Let $y_1^0=\frac{1}{\sqrt{2}}, y_2^0=\frac{1}{\sqrt{2}}, y_3^0=0.$ 
Fig. \ref{fig4} shows that the proposed stochastic 
semi-explicit symplectic schemes and stochastic midpoint scheme are convergent with mean-square order 1. 
Compared with $\widetilde F_{\Delta t}$ derived from \eqref{split1} and the midpoint scheme, the stochastic semi-explicit scheme $\widetilde F_{\Delta t}$ based on \eqref{strang} is the most accurate, and its efficiency is almost as high as that of $\widetilde F_{\Delta t}$ with \eqref{split1}, which can be seen in Table \ref{table4}.  
Moreover, as seen in Fig. \ref{fig5}, in terms of preserving of the Casimir function, the proposed stochastic semi-explicit symplectic scheme is superior to the stochastic midpoint scheme.

\begin{figure}
	\centering
	\subfigure{
		\begin{minipage}{12.5cm}
\centering
\includegraphics[height=3.5cm,width=4cm]{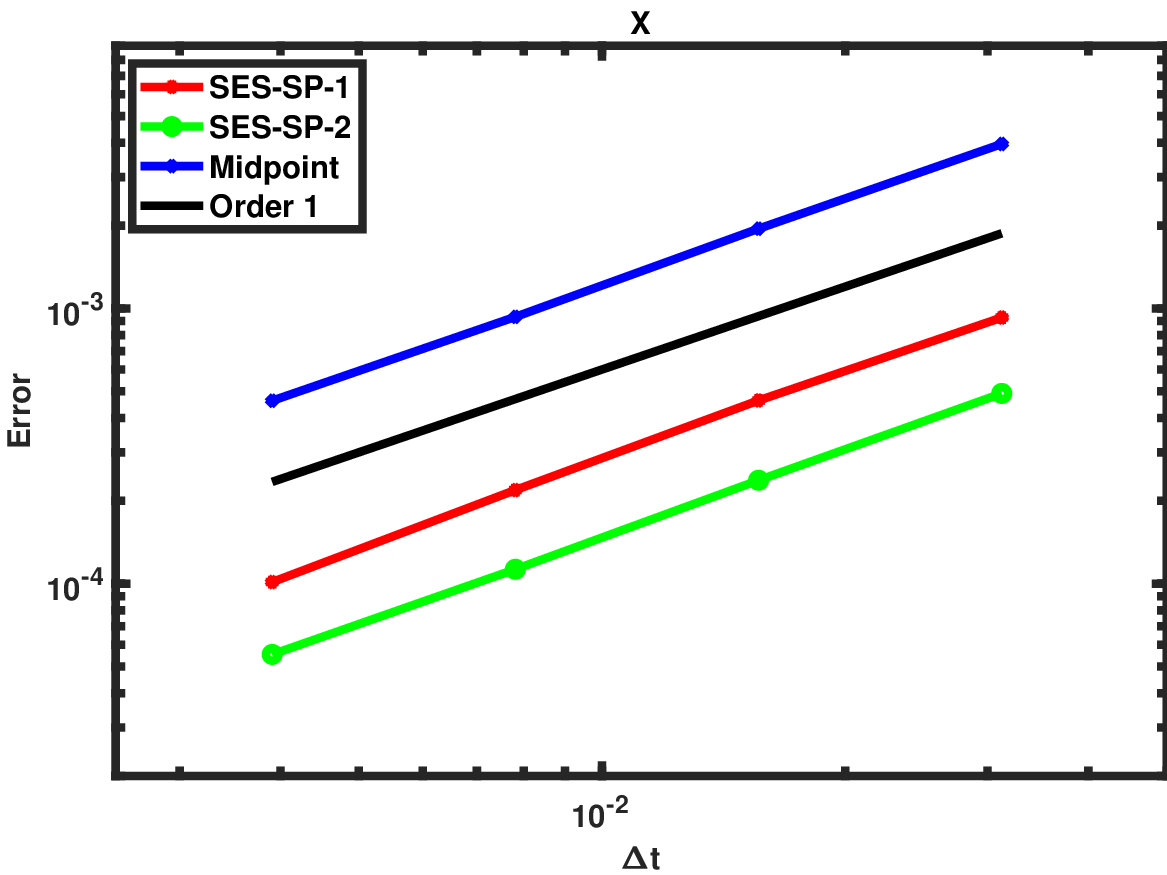}
\includegraphics[height=3.5cm,width=4cm]{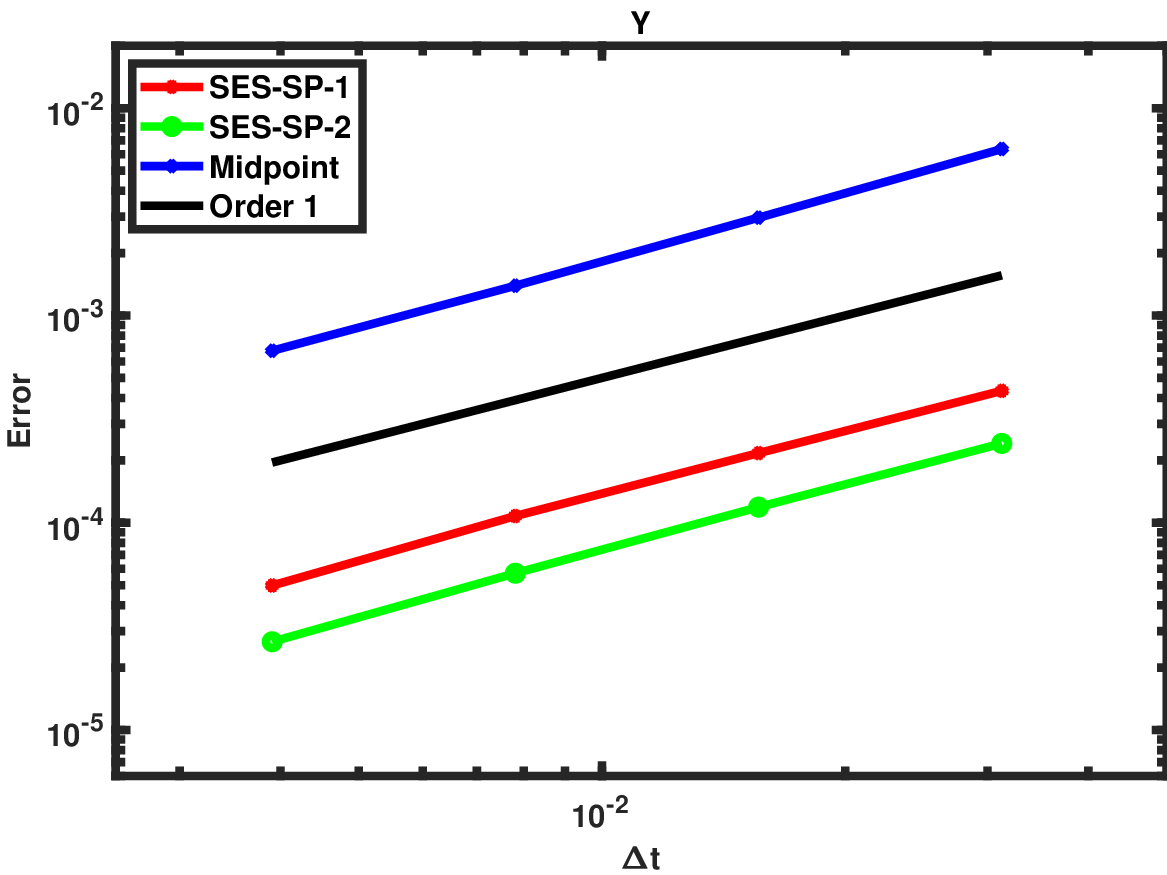}
\end{minipage}
	}
	\caption{Mean-square convergence order of various schemes in time with $T=1, c = 1, \gamma = 0.2$.} 
	\label{fig4}
\end{figure}

\begin{table}  
\setlength{\abovecaptionskip}{0.1cm}
	\setlength{\belowcaptionskip}{-0.cm}
	\caption{Numerical errors and CPU time of three numerical schemes for Example 4 with $c=0.1, \gamma = 0.5$.}
	\centering
	\begin{tabularx}{\textwidth}{p{0.9cm}<{\centering}|p{1.7cm}p{1.7cm}<{\centering}p{1.7cm}<{\centering}|p{1.5cm}<{\centering}p{1.5cm}<{\centering}p{1.5cm}<{\centering}}
		\toprule
		 & \multicolumn{1}{c}{}  & \multicolumn{2}{c|}{Errors} & \multicolumn{3}{c}{CPU time} \\
		\toprule
		$\Delta t$ & SES-SP-1 &SES-SP-2 & Midpoint &SES-SP-1 &SES-SP-2 & Midpoint  \\
		\midrule
		$1/2^8$    &  1.2459e-04  & 6.2235e-05 &7.2964e-04& 18.55s & 19.29s &  33.78s  \\
		$1/2^{10}$   &  2.9374e-05  & 1.4868e-05 &1.9000e-04& 80.50s & 81.48s& 132.71s  \\
		$1/2^{12}$ &  7.4610e-06  & 3.8570e-06 &4.7823e-05& 310.99s & 312.04s & 511.69s \\
		$1/2^{13}$ &  3.7182e-06  & 1.8666e-06 &2.2612e-05& 616.06s & 617.13s & 979.68s \\
		\bottomrule
	\end{tabularx}
	\label{table4}	
\end{table}

\begin{figure}
	\centering
	\subfigure{
		\begin{minipage}{12.5cm}
\centering
\includegraphics[height=3.5cm,width=4cm]{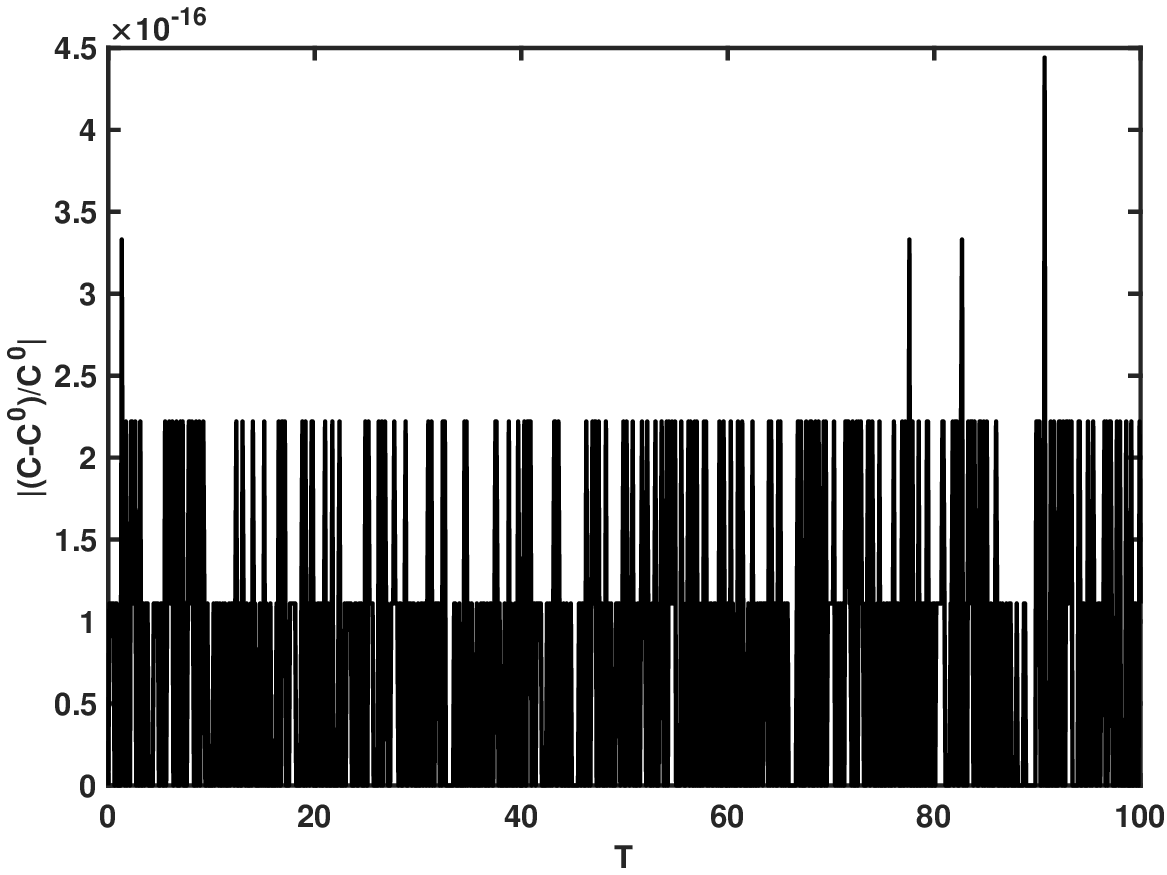}
\includegraphics[height=3.5cm,width=4cm]{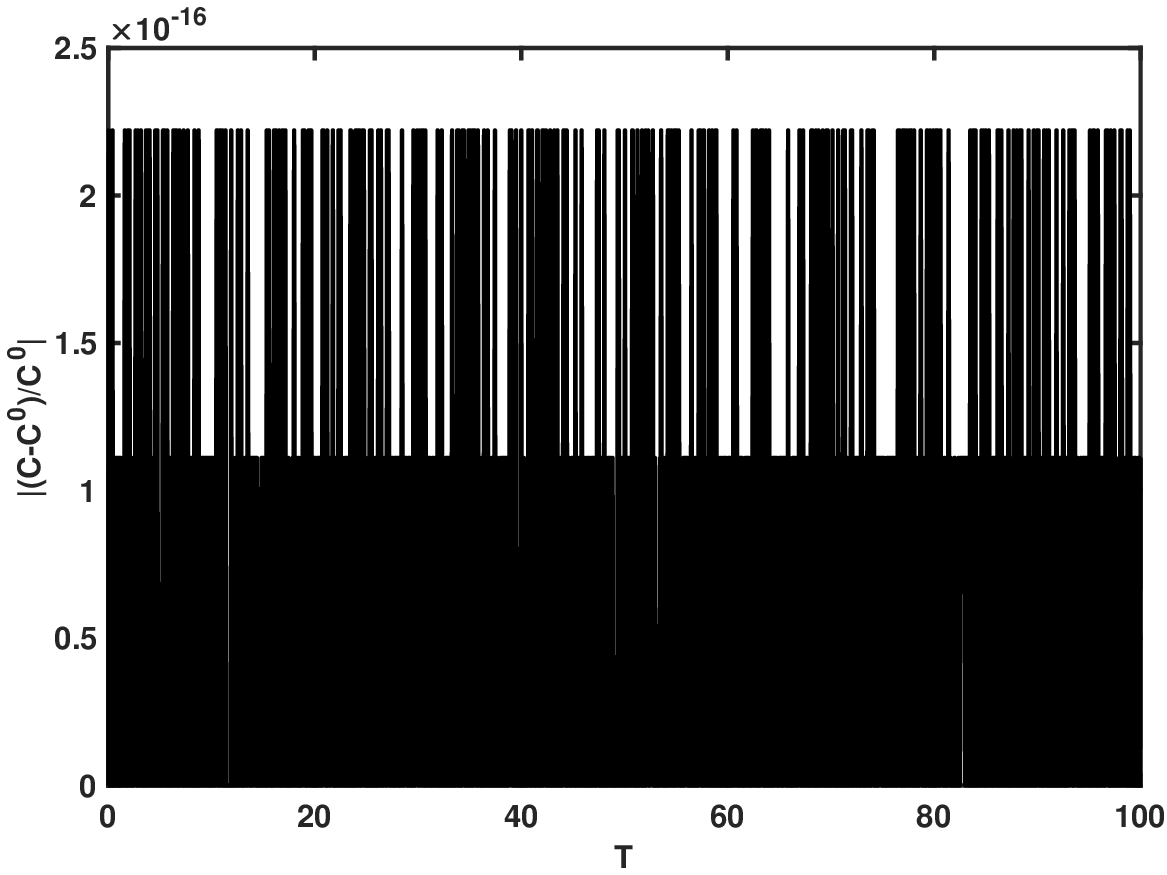}
\includegraphics[height=3.5cm,width=4cm]{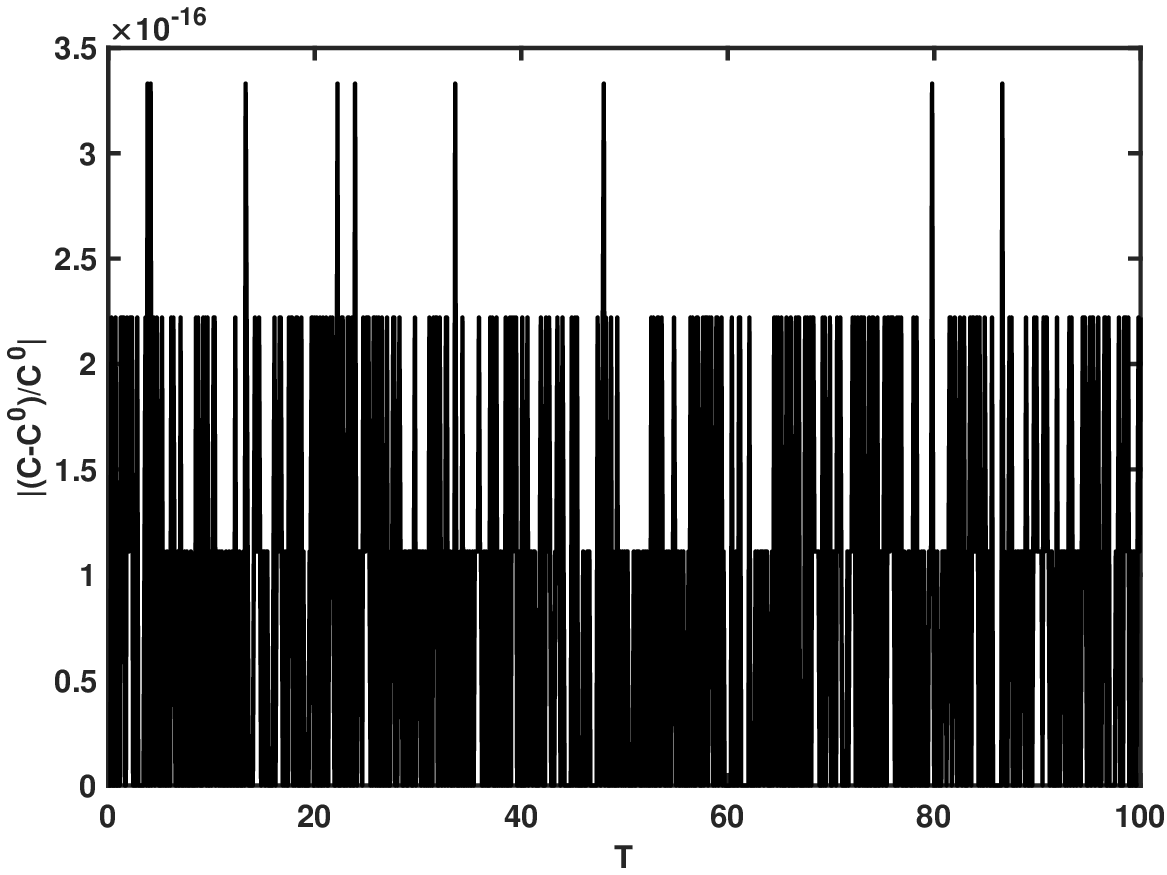}
\end{minipage}
	}
	\caption{Relative Casimir function with $T =100, \Delta t = 1\times 10^{-1}, c = 1, \gamma = 0.5 $: (1)SES-SP-1, (2) SES-SP-2, (3)Midpoint.} 
	\label{fig5}
\end{figure}
\end{ex}

\section{Stochastic semi-explicit multi-symplectic schemes for stochastic cubic Schr\"odinger equations with multiplicative noise}
The stochastic cubic Schr\"odinger equation is a fundamental SPDE in quantum mechanics, which describes the evolution of a quantum system subject to random fluctuations. 
This section is concerned with the stochastic cubic Schr\"odinger equation with multiplicative noise
\begin{equation}
\left\{
\label{eq;Schro}
\begin{aligned}
&\mathbf{i} du +  u_ {xx} dt + |u|^2 u dt = u\circ dW(t),\quad t\in(0,T),\\
& u(0)=u_0, 
\end{aligned}
\right.
\end{equation}
where $\mathbf{i}^2=-1,$ $u_{xx}$ denotes  the second derivative of $u$ on $\mathcal O=[0,L]$ with $\mathcal O$ being a bounded domain with homogeneous Dirichlet boundary condition. 
Here, $W(t)$ is a $\bf Q$-Wiener process, i.e., $W(t)=\sum\limits_{k\in \mathbb N^+}{\bf Q}^{\frac 12}e_k\beta_k(t)$ with $\{\beta_k\}_{k\in\mathbb N^+}$ being a sequence of independent Brownian motions on a probability space $(\Omega,$ 
$\mathcal F, (\mathcal F_t)_{t\ge 0},\mathbb P).$ 
Denoting $p$ (resp., $q$)  the real (resp., imaginary) part of the solution $u,$ it holds that 
\begin{equation*}
\left\{\begin{aligned}
&d q  = (p_{xx} +(p^2+q^2)p)dt  - p\circ dW(t), \quad\, q(0)=q_0,\\
&d p  = -(q_{xx} +(p^2+q^2)q)dt  + q\circ dW(t),\quad p(0)=p_0,
\end{aligned}
\right.
\end{equation*}
which preserves charge conservation law $$p(t)^2+q(t)^2=p_0^2+q_0^2,\quad a. s.,$$ and the infinite-dimensional stochastic symplectic structure
$
\int_{\mathcal O}\mathrm{d} p(t) \wedge \mathrm{d} q(t)dx=\int_{\mathcal O}\mathrm{d} p_0 \wedge \mathrm{d} q_0dx$ 
almost surely for any $t \in[0, T]$. Furthermore, introducing derivatives $\theta:=p_{x}$ and $\eta:=q_{x}$  
and defining a state variable $z=(p, q, \theta, \eta)^{\top}$, we reformulate the multi-symplectic form of \eqref{eq;Schro} as follows
\begin{equation}
\label{eq;Sch_multisym}
M d z+K z_x d t=\nabla S_1(z) d t+\nabla S_2(z) \circ d W(t)
\end{equation}
with
$$
\begin{aligned}
& M=\left[\begin{array}{cccc}
0 & -1 & 0 & 0 \\
1 & 0 & 0 & 0 \\
0 & 0 & 0 & 0 \\
0 & 0 & 0 & 0
\end{array}\right], \quad K=\left[\begin{array}{cccc}
0 & 0 & 1 & 0 \\
0 & 0 & 0 & 1 \\
-1 & 0 & 0 & 0 \\
0 & -1 & 0 & 0
\end{array}\right], \\
&
S_1(z)=-\frac{1}{4}\left(p^2+q^2\right)^2-\frac{1}{2}\left(\theta^2+\eta^2\right), \quad S_2(z)=\frac{1}{2}\left(p^2+q^2\right).
\end{aligned}
$$
It can be derived that \eqref{eq;Sch_multisym}
possesses the multi-symplectic conservation law
\begin{equation}
d[\mathrm{d} q \wedge \mathrm{d} p] +\partial_{x}[\mathrm{d} p \wedge \mathrm{d} \theta + \mathrm{d} q \wedge \mathrm{d} \eta]dt=0. \quad a. s.
\end{equation}

Define the discrete grid $\Omega_h := \{ x_i| 1 \leq i \leq L-1\}$ with $L\in \mathbb N^{+}\backslash \{1\},$  
and matrices $D^+$ and $D^-$ of order $(L-1)\times (L-1)$ as
\begin{align*}
D^+=\frac{1}{h}{
	\left[ \begin{array}{cccccc}
	-1&1&0&\cdots&0&0\\[0.05in]
	0&-1&1&\cdots&0&0\\
	\vdots&\vdots&\vdots&&\vdots&\vdots\\
	0&0&0&\cdots&-1&1\\[0.05in]
	0&0&0&\cdots&0&-1
	\end{array}
	\right ]},\quad
D^-=\frac{1}{h}{
	\left[ \begin{array}{cccccc}
	1&0&0&\cdots&0&0\\[0.05in]
	-1&1&0&\cdots&0&0\\
	\vdots&\vdots&\vdots&&\vdots&\vdots\\
	0&0&0&\cdots&1&0\\[0.05in]
	0&0&0&\cdots&-1&1
	\end{array}
	\right ]},
\end{align*}	
where $h=x_{i+1}-x_i,$ $i\in\{0,1,\ldots,L-1\}.$ 
Let $P(t):= \left(P_1(t),  \dots,P_{L-1}(t)\right)^\top,$ $Q(t):= \left(Q_1(t),\dots,Q_{L-1}(t)\right)^\top,$ 
where $P_i(t), Q_i(t)$ are approximations of $p(x_i,t),$ $q(x_i,t),$  respectively, for $1\leq i \leq L-1$ and $t\in[0,T].$ 
Denoting $\Theta:=D^- P=(\Theta_1,\ldots,\Theta_{L-1})^\top$ and $ B:=D^- Q=(B_1,\ldots,B_{L-1}),$ we have the spatial semi-discretization 
\begin{equation}\label{eq;semi_multisym}
\left\{\begin{aligned}
&d Q - D^+\Theta dt = \left(\text{diag}(P)^{2}+\text{diag}(Q)^{2}\right)Pdt  - \text{diag}(P)\mbf E\boldsymbol{\Lambda}\circ d\boldsymbol{\beta}(t),\\
&d P +  D^+B dt = -\left(\text{diag}(P)^{2}+\text{diag}(Q)^{2}\right)Qdt +\text{diag}(Q)\mbf E\boldsymbol{\Lambda}\circ d\boldsymbol{\beta}(t),\\
&D^- P = \Theta,\quad D^- Q = B,
\end{aligned}\right.
\end{equation}
where $\Theta_i(t),$ $ B_i(t)$ are approximations of $\theta(x_i,t)$, $ \eta(x_i,t)$, respectively, for $1\leq i \leq L-1$ and $t\in[0,T].$ 
Moreover,  
$\boldsymbol{\Lambda}= {\rm{diag} }\{ \sqrt{\lambda_1}, \sqrt{\lambda_2}, \dots,\sqrt{\lambda_M} \}$, $\boldsymbol {\beta}=\left(\beta_1, \beta_2, \dots,\beta_M\right)^\top,$ and 
\begin{equation*}
\mbf E={
	\left[ \begin{array}{cccc}
	e_1(x_1)&e_2(x_1)&\cdots&e_M(x_1)\\
	e_1(x_2)&e_2(x_2)&\cdots&e_M(x_2)\\
	\vdots&\vdots&&\vdots\\
	e_1(x_{L-1})&e_2(x_{L-1})&\cdots&e_M(x_{L-1})\\
	\end{array}
	\right ]_{(L-1) \times M}}.
\end{equation*}	
We note that \eqref{eq;semi_multisym} preserves the invariant  
$
I(P, Q):=\sum\limits_{i=1}^N\left(P_i^2+Q_i^2\right),
$
and 
\begin{align*}
&d(\mathrm{d}Q  \wedge \mathrm{d}P)+(\mathrm{d}Q  \wedge D^+ \mathrm{d}B
+\mathrm{d}P  \wedge D^+ \mathrm{d}\Theta +D^-\mathrm{d}Q  \wedge \mathrm{d}B
+D^-\mathrm{d}P  \wedge  \mathrm{d}\Theta )dt=0
\end{align*}
almost surely. 
In addition, by introducing $D^- Y = F, D^- X = G$,   
\eqref{eq;semi_multisym} can be rewritten as the following equivalent form
\begin{equation} \label{eq;extend_schrodinger_multisymplectic}
\left\{\begin{aligned}
	&d Q = D^+ \Theta dt + \left(\text{diag}(P)^{2}+\text{diag}(X)^{2}\right)Pdt  - \text{diag}(P)\mbf E\boldsymbol{\Lambda}\circ d\boldsymbol{\beta}(t),\\
	&d X = D^+ F dt +\left(\text{diag}(Y)^{2}+\text{diag}(Q)^{2}\right)Ydt - \text{diag}(Y)\mbf E\boldsymbol{\Lambda}\circ d\boldsymbol{\beta}(t),\\
	&d P = -D^+ B dt -\left(\text{diag}(Q)^{2}+\text{diag}(Y)^{2}\right)Qdt +\text{diag}v(Q)\mbf E\boldsymbol{\Lambda}\circ d\boldsymbol{\beta}(t),\\
	&d Y = -D^+ G dt - \left(\text{diag}(X)^{2}+\text{diag}(P)^{2}\right)Xdt+\text{diag}(X)\mbf E\boldsymbol{\Lambda}\circ d\boldsymbol{\beta}(t),\\
	&D^-P = \Theta,\quad  D^-Q = B,\quad D^-Y = F,\quad D^-X = G,
	\end{aligned}\right.
 \end{equation}
which satisfies the following conservation law
\begin{align*}
&d(\mathrm{d}Q  \wedge \mathrm{d}P + \mathrm{d}X \wedge \mathrm{d}Y)\\
&+(\mathrm{d}Q  \wedge D^+ \mathrm{d}B
+\mathrm{d}P  \wedge D^+ \mathrm{d}\Theta + \mathrm{d}X  \wedge D^+ \mathrm{d}G + \mathrm{d}Y  \wedge D^+ \mathrm{d}F)dt \\
&+(D^-\mathrm{d} Q  \wedge \mathrm{d}B
+D^-\mathrm{d} P  \wedge \mathrm{d}\Theta + D^-\mathrm{d} X  \wedge \mathrm{d}G + D^-\mathrm{d} Y  \wedge  \mathrm{d}F)dt =0, \quad a.s.
\end{align*} 
Under the same initial value, \eqref{eq;extend_schrodinger_multisymplectic} presents two copies of \eqref{eq;semi_multisym}, i.e., $(Q(t),P(t))=(X(t),Y(t)).$ 
Based on the separable feature, \eqref{eq;extend_schrodinger_multisymplectic} can be split into two subsystems on $[t_n, t_{n+1}],$
\begin{equation}
\label{subsystem1}
 \left\{\begin{aligned}
	&d X = D^+ F dt +\left(\diag(Y)^{2}+\diag(Q)^{2}\right)Ydt - \diag(Y)\mbf E\boldsymbol{\Lambda}\circ d\boldsymbol{\beta}(t),\\
	&d P = -D^+ B dt -\left(\diag(Q)^{2}+\diag(Y)^{2}\right)Qdt + \diag(Q)\mbf E\boldsymbol{\Lambda}\circ d\boldsymbol{\beta}(t),\\
	&d Q = 0,\quad d Y = 0, \quad D^-P = \Theta,\quad  D^-Q = B,\quad D^-Y = F,\quad D^-X = G,
\end{aligned}\right.
 \end{equation}
 and
 \begin{equation}
 \label{subsystem2}
\left\{\begin{aligned}
	&d Q = D^+ \Theta dt + \left(\diag(P)^{2}+\diag(X)^{2}\right)Pdt  - \diag(P)\mbf E\boldsymbol{\Lambda}\circ d\boldsymbol{\beta}(t),\\
	&d Y = -D^+ G dt - 
\left(\diag(X)^{2}+\diag(P)^{2}\right)Xdt+\diag(X)\mbf E\boldsymbol{\Lambda}\circ d\boldsymbol{\beta}(t),\\
	&d X = 0,\quad d P = -0,\quad D^-P = \Theta,\quad  D^-Q = B,\quad D^-Y = F,\quad D^-X = G.
\end{aligned}\right.
\end{equation}
whose exact solutions can be explicitly obtained.  
Denote $\mathcal{F}_{\Delta t}^a$ and $\mathcal{F}_{\Delta t }^b$ by the stochastic flow of \eqref{subsystem1} and \eqref{subsystem2}, respectively, and $\widetilde{\zeta}^n=((\widetilde{Q}^n)^\top, (\widetilde{X}^n)^\top,(\widetilde{P}^n)^\top,(\widetilde{Y}^n)^\top,(\widetilde{\Theta}^n)^\top,
(\widetilde{B}^n)^\top, (\widetilde{F}^n)^\top, 
(\widetilde{G}^n)^\top
)^\top$ for $n\in\{0,1,\ldots,N\}.$ 
Then we have a full discretization based on  integrator $\mathcal{F}_{\Delta t} =\mathcal{F}_{\Delta t/2}^a \star \mathcal{F}_{\Delta t }^b \star \mathcal{F}_{\Delta t/2}^a$, i.e.,
\begin{equation}
\label{eq;schrodinger_an}
\left\{\begin{aligned}
&Q_a= \widetilde{Q}^n,\quad Y_a= \widetilde{Y}^n,\quad B_a=\widetilde{B}^n=
D^-\widetilde{Q}^n, \quad F_a=\widetilde{F}^n=D^-\widetilde{Y}^n,\\
&P_a = \widetilde{P}^n -  \frac{\Delta t}{2}(D^+B_a +
(\diag(\widetilde{Q}^n)^2+ \diag(\widetilde{Y}^n)^2)\widetilde{Q}^n)
+\diag(\widetilde{Q}^n)\mathbf E\mathbf\Lambda\Delta\boldsymbol{\beta}_n^1,\\
&X_a = \widetilde{X}^n +  \frac{\Delta t}{2}(D^+F_a +
(\diag(\widetilde{Q}^n)^2+ \diag(\widetilde{Y}^n)^2)\widetilde{Y}^n)
-\diag(\widetilde{Y}^n)\mathbf E\mathbf\Lambda\Delta\boldsymbol{\beta}_n^1,\\
&\Theta_a=D^-P_a, \quad G_a=D^-X_a.
\end{aligned}\right.
\end{equation}
\vspace{-0.5em}
\begin{equation}
\label{eq;schrodinger_bn}
\left\{\begin{aligned}
&P_b = P_a,\quad X_b = X_a,\quad \Theta_b=D^-P_b, \quad G_b=D^-X_b, \\
&Q_b=Q_a+\Delta t(D^+\Theta_b+
(\diag(P_a)^2+ \diag(X_a)^2)P_a)
-\diag(P_a)\mathbf E\mathbf\Lambda\Delta\boldsymbol{\beta}_n,\\
&Y_b= Y_a -\Delta t(D^+G_b+
(\diag(P_a)^2+ \diag(X_a)^2)X_a)
+\diag(X_a)\mathbf E\mathbf\Lambda\Delta\boldsymbol{\beta}_n,\\
&B_b=D^-Q_b, \quad F_b=D^-Y_b,
\end{aligned}\right.
\end{equation}
\vspace{-0.5em}
\begin{equation}
\label{eq;schrodinger_a(n+1)}
\left\{\begin{aligned}
&\widetilde{Q}^{n+1}= Q_b,\quad \widetilde{Y}^{n+1}= Y_b, \quad \widetilde{B}^{n+1}=D^-\widetilde{Q}^{n+1}, \quad \widetilde{F}^{n+1}=D^-\widetilde{Y}^{n+1}, \\
&\widetilde{P}^{n+1} = P_b -  
\frac{\Delta t}{2}(D^+\widetilde{B}^{n+1}
+(\diag(\widetilde{Q}^{n+1})^2+ \diag(\widetilde{Y}^{n+1})^2)\widetilde{Q}^{n+1})\\
&\quad\quad\quad
+\diag(\widetilde{Q}^{n+1})\mathbf E\mathbf\Lambda\Delta\boldsymbol{\beta}_n^2,\\
&\widetilde{X}^{n+1} = X_b +  \frac{\Delta t}{2}
(D^+\widetilde{F}^{n+1} +
(\diag(\widetilde{Q}^{n+1})^2+ \diag(\widetilde{Y}^{n+1})^2)\widetilde{Y}^{n+1})\\
&\quad\quad\quad
-\diag(\widetilde{Y}^{n+1})\mathbf E\mathbf\Lambda\Delta\boldsymbol{\beta}_n^2,\\
&\widetilde{\Theta}^{n+1}=D^-\widetilde P^{n+1}, \quad \widetilde{G}^{n+1}=D^-\widetilde{X}^{n+1},
\end{aligned}\right.
\end{equation}
where $ \Delta\boldsymbol{\beta}_n=\boldsymbol{\beta}(t_{n+1})-\boldsymbol{\beta}(t_{n}),$ $\Delta\boldsymbol{\beta}_n^2=
\boldsymbol{\beta}(t_{n+1})-\boldsymbol{\beta}(t_n+\frac {\Delta t}2),$ and $\Delta\boldsymbol{\beta}_n^1=
\boldsymbol{\beta}(t_n+\frac {\Delta t}2)-\boldsymbol{\beta}(t_{n})$ for $n\in\{0,1,\ldots,N-1\}.$  
It can be derived that the full discretization preserves the discrete stochastic multi-symplectic conservation law, that is, for $n\in\{0,1,\ldots,N-1\},$ 
\begin{align*}
&\mathrm{d}\widetilde{Q}^{n+1} \wedge \mathrm{d}\widetilde{P}^{n+1}+ \mathrm{d}\widetilde{X}^{n+1} \wedge \mathrm{d}\widetilde{Y}^{n+1}+\Delta t(\mathrm{d}\widetilde{Q}^{n+1} \wedge D^+ \mathrm{d}\widetilde{B}^{n+1}\\
&+
\mathrm{d}\widetilde{P}^{n+1}  \wedge D^+ \mathrm{d}\widetilde{\Theta}^{n+1}+\mathrm{d}\widetilde{X}^{n+1}  \wedge D^+ \mathrm{d}\widetilde{G}^{n+1} + \mathrm{d}\widetilde{Y}^{n+1}  \wedge D^+ \mathrm{d}\widetilde{F}^{n+1})\\
&+\Delta t(D^-\mathrm{d}\widetilde{Q}^{n+1} \wedge  \mathrm{d}\widetilde{B}^{n+1}+
D^-\mathrm{d}\widetilde{P}^{n+1}  \wedge  \mathrm{d}\widetilde{\Theta}^{n+1})\\
&+\Delta t(D^-\mathrm{d}\widetilde{X}^{n+1}  \wedge D^- \mathrm{d}\widetilde{G}^{n+1} +  D^-\mathrm{d}\widetilde{Y}^{n+1}  \wedge \mathrm{d}\widetilde{F}^{n+1})\\
=&\mathrm{d}\widetilde{Q}^{n} \wedge \mathrm{d}\widetilde{P}^{n}+ \mathrm{d}\widetilde{X}^{n} \wedge \mathrm{d}\widetilde{Y}^{n}+\Delta t(\mathrm{d}\widetilde{Q}^{n} \wedge D^+ \mathrm{d}\widetilde{B}^{n}\\
&+
\mathrm{d}\widetilde{P}^{n}  \wedge D^+ \mathrm{d}\widetilde{\Theta}^{n}+ \mathrm{d}\widetilde{X}^{n}  \wedge D^+ \mathrm{d}\widetilde{G}^{n} + \mathrm{d}\widetilde{Y}^{n}  \wedge D^+ \mathrm{d}\widetilde{F}^{n})\\
&+\Delta t(D^-\mathrm{d}\widetilde{Q}^{n} \wedge  \mathrm{d}\widetilde{B}^{n}+
D^-\mathrm{d}\widetilde{P}^{n}  \wedge  \mathrm{d}\widetilde{\Theta}^{n})\\
&+\Delta t(D^-\mathrm{d}\widetilde{X}^{n}  \wedge D^- \mathrm{d}\widetilde{G}^{n} +  D^-\mathrm{d}\widetilde{Y}^{n}  \wedge \mathrm{d}\widetilde{F}^{n}),\quad a.s.
\end{align*} 
\vspace{-1.5em}
\begin{rem}
We would like to mention that $\widetilde B^n,$ $\widetilde F^n,$ $\widetilde \Theta^n$ and $\widetilde G^n$ are determined by $\widetilde Q^n,$ $\widetilde Y^n,$ $\widetilde P^n$ and $\widetilde X^n$, respectively. As a result, $\mathcal F_{\Delta t}$ can be seen as a mapping from $(\widetilde Q^n,\widetilde Y^n,\widetilde P^n,\widetilde X^n)$ to $(\widetilde Q^{n+1},\widetilde Y^{n+1},\widetilde P^{n+1},\widetilde X^{n+1}),$ where $n\in\{0,1,\ldots,N-1\}.$ 
Moreover, the fully discrete scheme based on the composition $\mathcal{F}_{\Delta t} =\mathcal{F}_{\Delta t}^a \star \mathcal{F}_{\Delta t }^b$ or $\mathcal{F}_{\Delta t}^b \star \mathcal{F}_{\Delta t }^a$ or $\mathcal{F}_{\Delta t/2}^b \star \mathcal{F}_{\Delta t }^a \star \mathcal{F}_{\Delta t/2}^b$ also preserves the discrete stochastic multi-symplectic conservation law. 
\end{rem}
\begin{rem}
Set augmented Hamiltonians 
\begin{align*}
&\tilde H_1(Q,X,P,Y)=\sum\limits_{i}^{L-1}(P_i^2+X_i^2)^2+
\sum\limits_{i}^{L-1}(Y_i^2+Q_i^2)^2+\gamma_1\|P-Y\|^2+\gamma_1\|Q-X\|^2,\\
&\tilde H_2(Q,X,P,Y)=-\sum\limits_{i}^{L-1}(P_i^2+X_i^2+Y_i^2+Q_i^2)+\gamma_2\|P-Y\|^2+\gamma_2\|Q-X\|^2,
\end{align*}
where $\gamma_1,\gamma_2\in\mathbb R.$ Then 
\eqref{eq;semi_multisym} possesses another equivalent form
\begin{align*}
	&d Q = D^+ \Theta dt+ (\diag(P)^{2}+\diag(X)^{2})Pdt +2\gamma_1(P-Y)dt\\
&\quad\quad\quad - \diag(P+2\gamma_2(P-Y))\mbf E\boldsymbol{\Lambda}\circ d\boldsymbol{\beta}(t),\\
	&d X = D^+ F dt +\left(\diag(Y)^{2}+\diag(Q)^{2}\right)Ydt-2\gamma_1(P-Y)dt \\
 &\quad\quad\quad - \diag(Y-2\gamma_2(P-Y))\mbf E\boldsymbol{\Lambda}\circ d\boldsymbol{\beta}(t),\\
	&d P = -D^+ B dt -\left(\diag(Q)^{2}+\diag(Y)^{2}\right)Qdt-2\gamma_1(Q-X)dt\\
 &\quad\quad\quad + \diag(Q-2\gamma_2(Q-X))\mbf E\boldsymbol{\Lambda}\circ d\boldsymbol{\beta}(t),\\
	&d Y = -D^+ G dt - \left(\diag(X)^{2}+\diag(P)^{2}\right)Xdt+2\gamma_1(Q-X)dt\\
&\quad\quad\quad+\diag(X+2\gamma_2(Q-X))\mbf E\boldsymbol{\Lambda}\circ d\boldsymbol{\beta}(t),\\
	&D^-P = \Theta,\quad  D^-Q = B,\quad D^-Y = F,\quad D^-X = G,
	\end{align*}
which can be split into three explicitly solvable subsystems \eqref{subsystem1},  \eqref{subsystem2} and 
\begin{equation}
\label{subsystem3}
\left\{\begin{aligned}
	&d Q =2\gamma_1(P-Y)dt - \diag(2\gamma_2(P-Y))\mbf E\boldsymbol{\Lambda}\circ d\boldsymbol{\beta}(t),\\
	&d X = -2\gamma_1(P-Y)dt +\diag(2\gamma_2(P-Y))\mbf E\boldsymbol{\Lambda}\circ d\boldsymbol{\beta}(t),\\
	&d P =-2\gamma_1(Q-X)dt- \diag(2\gamma_2(Q-X))\mbf E\boldsymbol{\Lambda}\circ d\boldsymbol{\beta}(t),\\
	&d Y =2\gamma_1(Q-X)dt+\diag(2\gamma_2(Q-X))\mbf E\boldsymbol{\Lambda}\circ d\boldsymbol{\beta}(t),\\
	&D^-P = \Theta,\quad  D^-Q = B,\quad D^-Y = F,\quad D^-X = G.
	\end{aligned}\right.
 \end{equation}
Denoting $F^c_{\Delta t}$  by the stochastic flow of \eqref{subsystem3},  we can obtain the stochastic multi-symplectic scheme based on $
\mathcal{F}_{\Delta t}: =\mathcal{F}_{\Delta t }^a \star \mathcal{F}_{\Delta t }^b\star \mathcal{F}_{\Delta t }^c$ or $\mathcal{F}_{\Delta t / 2}^a \star \mathcal{F}_{\Delta t/2}^b\star \mathcal{F}_{\Delta t }^c
\star \mathcal{F}_{\Delta t/2}^b\star \mathcal{F}_{\Delta t / 2}^c.$ 
\end{rem}

However, the fully disreted scheme proposed above is multi-symplectic in the extended phase space but not in the original phase space. 
To this issue, we apply the same arguments as in Scheme 2.1 and let 
$(\widetilde{Q}^n,\widetilde{P}^n,\widetilde{B}^n,\widetilde{\Theta}^n)=
(\widetilde{X}^n,\widetilde{Y}^n,\widetilde{G}^n,\widetilde{F}^n)$ for each $n\in\{0,1,\ldots, N\}.$ 
Then, we define a mapping $\widetilde{\mathcal{F}}_{\Delta t}: (Q^n,P^n,B^n,\Theta^n)\rightarrow  (Q^{n+1},P^{n+1},B^{n+1},\Theta^{n+1})$ in the original phase space, which eliminates the problematic defect $(\widetilde{Q}^n-\widetilde{X}^n, \widetilde{P}^n-\widetilde{Y}^n)$.  

\noindent$\bf\hrulefill$

\noindent {\textbf {Scheme 3.1}} Semi-explicit multi-symplectic scheme for stochastic cubic Schr\"odinger equations with multiplicative noise

\vspace{0.5em}
\noindent Given $\xi^n =\left((Q^n)^\top,(P^n)^\top,(B^n)^\top,(\Theta^n)^\top\right)^\top$ with $B^{n}=D^-Q^{n}$ and $\Theta^{n}=D^-P^{n},$ for $n\in\{0,1,\ldots, N-1\},$ and an extended phase space symplectic integrator $\mathcal{F}_{\Delta t}$,   
\begin{itemize}
\item[1.] Set $Z^n:=\left((Q^n)^\top, (Q^n)^\top, (P^n)^\top, (P^n)^\top\right)^\top\in \mathcal N:=\ker(A)$ and $\widetilde{Z}^n:=Z^n+A^\top \lambda,$ where 
\begin{equation*}
A=\left[\begin{array}{cccc}
I_{L-1} & -I_{L-1} & 0 &0\\
0 & 0 & I_{L-1} & -I_{L-1} 
\end{array}\right],
\end{equation*}
and $\widetilde{Z}^n= \left((\widetilde{Q}^{n})^\top, (\widetilde{X}^{n})^\top, (\widetilde{P}^{n})^\top, (\widetilde{Y}^{n})^\top\right)^\top.$ 
\item[2.]  Find $\lambda \in \mathbb{R}^{2(L-1)}$ such that ${Z}^{n+1}:=\mathcal{F}_{\Delta t}(\widetilde{Z}^{n})+A^\top \lambda \in \mathcal{N}.$
\item [3.]Let $\widetilde{Z}^{n+1}=\mathcal F_{\Delta t}(\widetilde{Z}^{n})= \widetilde{Z}^n= \left((\widetilde{Q}^{n+1})^\top, (\widetilde{X}^{n+1})^\top, (\widetilde{P}^{n+1})^\top, (\widetilde{Y}^{n+1})^\top\right)^\top 
.$ \item [4.] Let $Z^{n+1}=\left(
(Q^{n+1})^\top, (Q^{n+1})^\top, (P^{n+1})^\top, (P^{n+1})^\top
\right)^\top:=\widetilde{Z}^{n+1}+A^\top \lambda.$ 
\item [5.] $\xi^{n+1}:=\left((Q^{n+1})^\top, (P^{n+1})^\top,(B^{n+1})^\top,(\Theta^{n+1})^\top\right)^\top$, where $B^{n+1}=D^-Q^{n+1}$ and $\Theta^{n+1}=D^-P^{n+1}.$ 
\end{itemize}

\noindent$\bf\hrulefill$

\begin{thm}
The full discretization 
$\widetilde{\mathcal{F}}_{\Delta t}: (Q^n,P^n,B^n,\Theta^n)\rightarrow  (Q^{n+1},P^{n+1},B^{n+1},\Theta^{n+1})$ preserves the stochastic multi-symplectic conservation law, i.e., 
\begin{align*}
&\mathrm{d}{Q}^{n+1} \wedge \mathrm{d}{P}^{n+1} +\Delta t(\mathrm{d}{Q}^{n+1} \wedge D^+ \mathrm{d}{B}^{n+1}\\
&+
\mathrm{d}{P}^{n+1}  \wedge D^+ \mathrm{d}{\Theta}^{n+1}+D^-\mathrm{d}{Q}^{n+1} \wedge  \mathrm{d}{B}^{n+1}+
D^-\mathrm{d}{P}^{n+1}  \wedge  \mathrm{d}{\Theta}^{n+1})\\
=&\mathrm{d}{Q}^{n} \wedge \mathrm{d}{P}^{n} +\Delta t(\mathrm{d}{Q}^{n} \wedge D^+ \mathrm{d}{B}^{n}\\
&+
\mathrm{d}{P}^{n}  \wedge D^+ \mathrm{d}{\Theta}^{n}+D^-\mathrm{d}{Q}^{n} \wedge  \mathrm{d}{B}^{n}+
D^-\mathrm{d}{P}^{n}  \wedge  \mathrm{d}{\Theta}^{n}),\quad a.s.
\end{align*} 
for $n\in\{0,1,\ldots,N-1\}.$ 
\end{thm}
\begin{proof} 

Since $(D^-)^\top = -D^+,$ $\widetilde A:=D^+D^-$ is symmetric, and
\begin{align*}
&P_a = \widetilde{P}^n -  \frac{\Delta t}{2}(\widetilde A\widetilde Q^n +
(\diag(\widetilde{Q}^n)^2+ \diag(\widetilde{Y}^n)^2)\widetilde{Q}^n)
+\diag(\widetilde{Q}^n)\mathbf E\mathbf\Lambda\Delta\boldsymbol{\beta}_n^1,\\
&X_a=\widetilde{X}^n +  \frac{\Delta t}{2}(\tilde A\widetilde Y^n+
(\diag(\widetilde{Q}^n)^2+ \diag(\widetilde{Y}^n)^2)\widetilde{Y}^n)
-\diag(\widetilde{Y}^n)\mathbf E\mathbf\Lambda\Delta\boldsymbol{\beta}_n^1,
\end{align*}
where $n\in\{0,1,\ldots,N-1\},$ we obtain 
$$\mathrm{d}Q_a  \wedge \mathrm{d}P_a + \mathrm{d}X_a  \wedge \mathrm{d}Y_a=\mathrm{d}\widetilde{Q}^{n} \wedge \mathrm{d}\widetilde{P}^{n}+ \mathrm{d}\widetilde{X}^{n} \wedge \mathrm{d}\widetilde{Y}^{n}.$$ 

Repeating similar arguments as above, 
\begin{align*}
&\mathrm{d}\widetilde{Q}^{n+1}  \wedge \mathrm{d}\widetilde{P}^{n+1}  + \mathrm{d}\widetilde{X}^{n+1}   \wedge \mathrm{d}\widetilde{Y}^{n+1}\\
=&
\mathrm{d}Q_b  \wedge \mathrm{d}P_b + \mathrm{d}X_b  \wedge \mathrm{d}Y_b
=\mathrm{d}Q_a  \wedge \mathrm{d}P_a + \mathrm{d}X_a  \wedge \mathrm{d}Y_a.
\end{align*}
As a result, $\left(\widetilde{Q}^{n}, \widetilde{X}^{n}, \widetilde{P}^{n}, \widetilde{Y}^{n}\right)\to \left(\widetilde{Q}^{n+1}, \widetilde{X}^{n+1}, \widetilde{P}^{n+1}, \widetilde{Y}^{n+1}\right)$ is a symplectic transformation. 
Taking advantage of Theorem \ref{tm:finite symplectic}, we deduce 
$$\mathrm{d}Q^{n+1}  \wedge \mathrm{d}P^{n+1} =\mathrm{d}Q^{n}  \wedge \mathrm{d}P^{n}.$$
Since $B^{n}=D^-Q^{n}$ and $\Theta^{n}=D^-P^{n},$ we derive
\begin{align*}
\mathrm{d}{Q}^{n} \wedge D^+ \mathrm{d}{B}^{n}+
\mathrm{d}{P}^{n}  \wedge D^+ \mathrm{d}{\Theta}^{n}+D^-\mathrm{d}{Q}^{n} \wedge  \mathrm{d}{B}^{n}+
D^-\mathrm{d}{P}^{n}  \wedge  \mathrm{d}{\Theta}^{n}=0,
\end{align*}
which implies the result. 
\end{proof}

For the stochastic nonlinear Schr\"odinger equation, we let
$x\in [-5,5],$ $P(0) = cos(2x)sech(x),$ $Q(0)= sin(2x)sech(x)$. 
The basis $\left\{e_{k}\right\}_{k\in \mathbb{N}^+}$ and eigenvalue $\left\{q_{k}\right\}_{k\in \mathbb{N}+}$ of $\bf Q$ are chosen as
$e_{k}(x)=\frac{1}{\sqrt{5}} \sin (k \pi x)$ and $q_{k}=\frac{1}{k^{6}},$ respectively.  
From Fig. \ref{fig14} it can be found that the full discretization based on the semi-explicit symplectic scheme 
with $\mathcal{F}_{\Delta t} =\mathcal{F}_{\Delta t}^a \star \mathcal{F}_{\Delta t }^b$ or $\mathcal{F}_{\Delta t} =\mathcal{F}_{\Delta t/2}^a \star \mathcal{F}_{\Delta t }^b \star \mathcal{F}_{\Delta t/2}^a$ 
can preserve the discrete charge conservation law, which is the same as the stochastic midpoint scheme. 
\begin{figure}[h!]
	\centering
	\subfigure{
		\begin{minipage}{12.5cm}
\centering
\includegraphics[height=3.5cm,width=4cm]{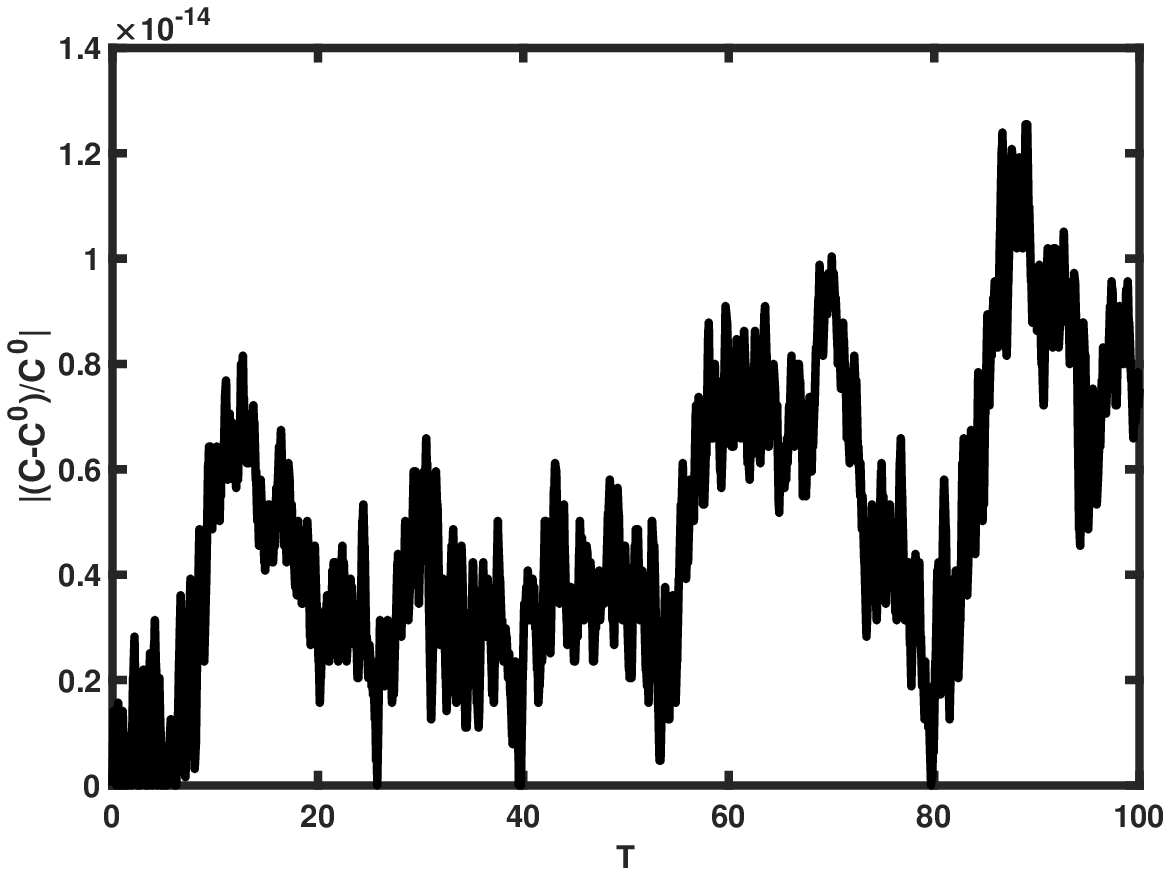}
\includegraphics[height=3.5cm,width=4cm]{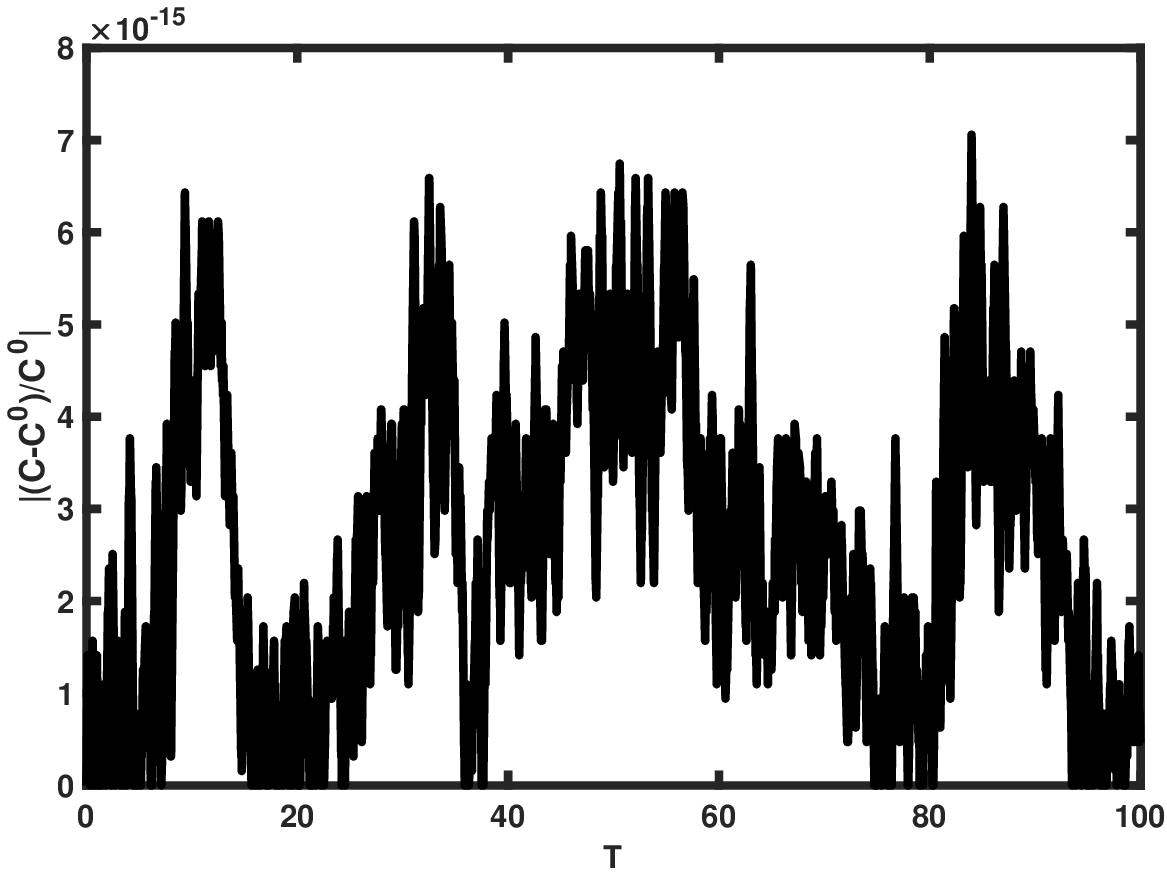}
\includegraphics[height=3.5cm,width=4cm]{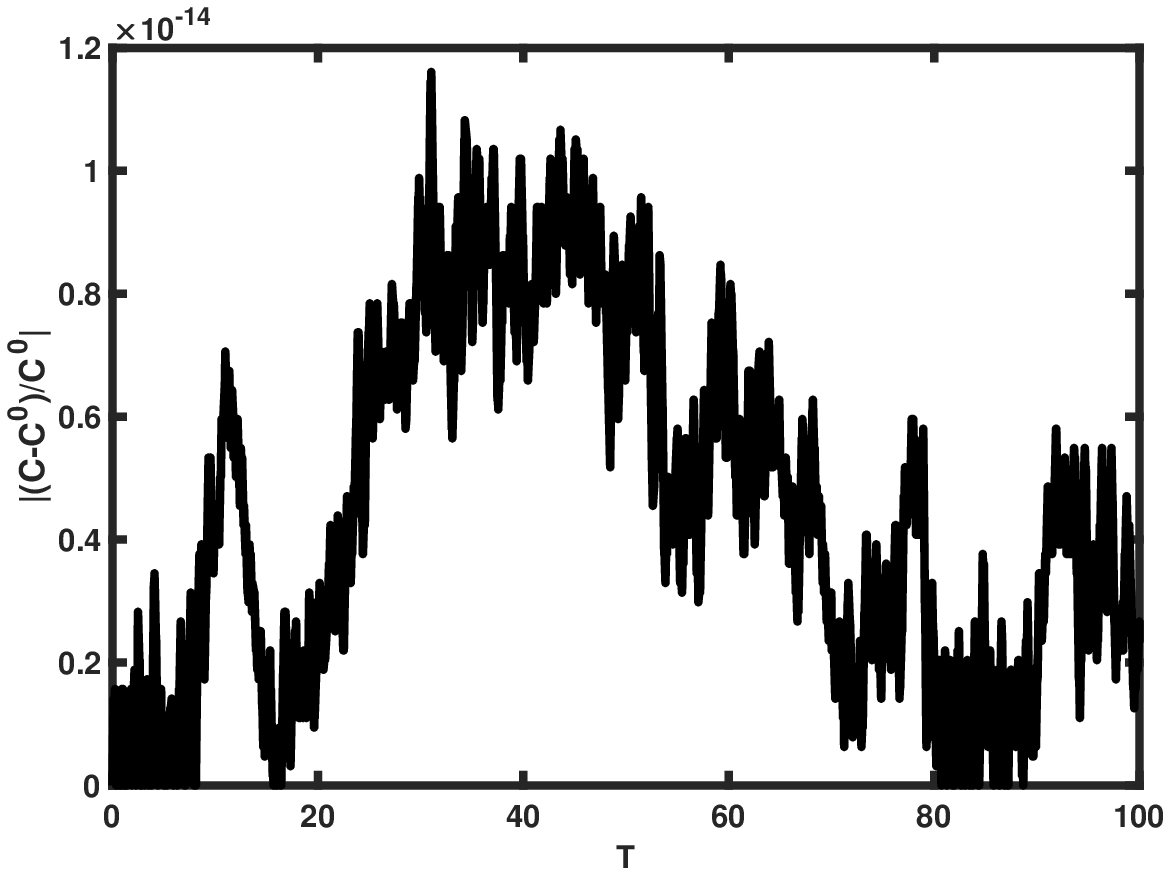}
\end{minipage}
	}
	\caption{Relative Charge with $T =100, \Delta t = 1\times 10^{-3}, h=1, \gamma = 0 $: (1)SES-SP-1, (2) SES-SP-2, (3)Midpoint.} 
	\label{fig14}
\end{figure}

Below we turn to proving the discrete charge conservation law of the proposed full discretization. 

\begin{thm}
The full discretization   
$\widetilde{\mathcal{F}}_{\Delta t}: (Q^n,P^n,B^n,\Theta^n)\rightarrow  (Q^{n+1},P^{n+1},B^{n+1},\Theta^{n+1})$ with $\mathcal{F}_{\Delta t} =\mathcal{F}_{\Delta t}^a \star \mathcal{F}_{\Delta t }^b$ or $\mathcal{F}_{\Delta t} =\mathcal{F}_{\Delta t/2}^a \star \mathcal{F}_{\Delta t }^b \star \mathcal{F}_{\Delta t/2}^a$ preserves the discrete charge conservation law, i.e., 
\begin{align*}
(P^{n+1})^\top P^{n+1}+(Q^{n+1})^\top Q^{n+1}=(P^{n})^\top P^{n}+(Q^{n})^\top Q^{n},\quad a.s.
\end{align*}
for $n\in\{0,1,\ldots,N-1\}.$ 
\end{thm}
\begin{proof}
From Scheme 3.1 it can be found that 
\begin{align*}
&\left(\widetilde Q^{n}, \widetilde X^{n}, \widetilde P^{n}, \widetilde Y^{n}\right)
=\left(Q^{n}+\lambda_1, Q^{n}-\lambda_1, P^{n}+\lambda_2, P^{n}-\lambda_2\right),\\
&\left(Q^{n+1}, Q^{n+1}, P^{n+1}, P^{n+1}\right)=\left(\widetilde{Q}^{n+1}+\lambda_1, \widetilde{X}^{n+1}-\lambda_1, \widetilde{P}^{n+1}+\lambda_2, \widetilde{Y}^{n+1}-\lambda_2\right).
\end{align*}
and 
\begin{equation}
\widetilde{X}^{n+1} - \widetilde{Q}^{n+1} = 2\lambda_1= \widetilde{Q}^{n} - \widetilde{X}^{n},
~~\widetilde{Y}^{n+1} - \widetilde{P}^{n+1} = 2\lambda_2= \widetilde{P}^{n} - \widetilde{Y}^{n}.
\end{equation}
As a consequence, 
\begin{align*}
&(Q^{n+1})^\top Q^{n+1}+(P^{n+1})^\top P^{n+1}\\
=& (\widetilde{Q}^{n+1}+\lambda_1)^\top(\widetilde{X}^{n+1}-\lambda_1)+(\widetilde{P}^{n+1}+\lambda_2)^\top(\widetilde{Y}^{n+1}-\lambda_2)\\
=&(\widetilde{Q}^{n+1})^\top\widetilde{X}^{n+1}+(\widetilde{P}^{n+1})^\top\widetilde{Y}^{n+1}+\lambda_1^2+\lambda_2^2,
\end{align*}
and 
\begin{align*}
(Q^{n})^\top Q^{n}+(P^{n})^\top P^{n}
=& (\widetilde{Q}^{n}-\lambda_1)^\top(\widetilde{X}^{n}+\lambda_1)+(\widetilde{P}^{n}-\lambda_2)^\top(\widetilde{Y}^{n}+\lambda_2)\\
=&(\widetilde{Q}^{n})^\top\widetilde{X}^{n}+(\widetilde{P}^{n})^\top\widetilde{Y}^{n}+\lambda_1^2+\lambda_2^2.
\end{align*}
According to $(\widetilde{Q}^{n+1})^\top\widetilde{X}^{n+1}+(\widetilde{P}^{n+1})^\top\widetilde{Y}^{n+1}=(\widetilde{Q}^{n})^\top\widetilde{X}^{n}+(\widetilde{P}^{n})^\top\widetilde{Y}^{n},$ we obtain the result.
\end{proof}

\section*{Acknowledgements}
This work is supported by National Natural Science Foundation of China (No. 12101596, No. 12031020,  No. 12171047, No. 12301518).


\end{document}